\documentclass{compositio}

\usepackage{amssymb,amsmath,amscd,amsthm,mathrsfs,graphicx} 
\usepackage[numbers]{natbib}
\usepackage{tikz}
\usepackage{blkarray}
\usepackage{xcolor}

\usepackage[pdftex,breaklinks,colorlinks]{hyperref} 
\hypersetup{linkcolor=blue,citecolor=red}
\usepackage{cleveref}
\usepackage{caption}

\newcommand\Tstrut{\rule{0pt}{2.5ex}}      
\newcommand\Bstrut{\rule[-1ex]{0pt}{0pt}}

\newtheorem{theorem}{Theorem}[section]
\newtheorem{lemma}[theorem]{Lemma}
\newtheorem{proposition}[theorem]{Proposition}
\newtheorem{corollary}[theorem]{Corollary}

\theoremstyle{definition}
\newtheorem{remark}[theorem]{Remark}
\newtheorem{definition}[theorem]{Definition}
\newtheorem{example}[theorem]{Example}
\newtheorem{problem}[theorem]{Problem}

\newtheorem{question}[theorem]{Question}

\newcommand{\mf}{m_{_{\mathrm{faithful}}}}
\newcommand{\g}{\mathfrak{g}}
\newcommand{\ZZ}{\mathrm{Z}}
\newcommand{\Z}{\mathbb{Z}}
\newcommand{\C}{\mathbb{C}}
\newcommand{\Q}{\mathbb{Q}}
\newcommand{\Hom}{\mathrm{Hom}}

\newcommand{\F}{\mathbb{F}}
\newcommand{\Tr}{\mathbf{Tr}}
\newcommand{\GL}{\mathrm{GL}}
\newcommand{\U}{{\mathrm{U}}}
\newcommand{\Span}{\mathrm{Span}}
\newcommand{\ad}{\mathrm{ad}}
\newcommand{\Iex}{I_\mathrm{ex}}

\newcommand{\rank}{\mathrm{rk}}
\newcommand{\f}{\mathfrak{f}}
\newcommand{\FZ}{\mathbb{Z}/p\mathbb{Z}}
\newcommand{\1}{\mathbf{1}}
\newcommand{\bb}{{\bf b}}
\newcommand{\R}{\mathrm{R}}
\newcommand{\e}{{\bf e}}

\newcommand{\ba}{\mathbf{a}}

\newcommand{\G}{\mathrm{G}}

\newcommand{\V}{\mathcal{V}}
\newcommand{\m}{\mathfrak{m}}
\newcommand{\OO}[1]{\mathcal{O}/\mathfrak{p}^{#1}}
\newcommand{\restr}[2]{#1_{|_{#2}}}
\newcommand{\GG}{\mathscr{G}}
\newcommand{\Hal}{\mathcal{H}}
\newcommand{\T}{\mathbf{T}}

\newcommand{\vv}{{\bf v}}
\newcommand{\ww}{{\bf w}}
\newcommand{\zz}{{\bf z}}
\newcommand{\uu}{{\bf u}}
\newcommand{\pr}{\mathsf{proj}}
\newcommand{\sfe}{\mathsf{e}}

\DeclareMathOperator{\Hei}{Heis}

\renewcommand{\contentsname}{Contents\\{\footnotesize\normalfont(A table
of contents should normally not be included)}}

\begin{document}

\title[Polynomiality of the faithful dimension of $p$-groups]{Kirillov's orbit method and polynomiality of the faithful dimension of $p$-groups}

\author[M. Bardestani]{Mohammad Bardestani}
\email{mohammad.bardestani@gmail.com}
\address{ DPMMS, Centre for Mathematical Sciences, Wilberforce Road, Cambridge, CB3 0WB.}

\author[K. Mallahi-Karai]{Keivan Mallahi-Karai}
\email{k.mallahikarai@jacobs-university.de }
\address{Jacobs University Bremen, Campus Ring I, 28759 Bremen, Germany.}

\author[H. Salmasian]{Hadi Salmasian}
\email{hadi.salmasian@uottawa.ca}
\address{Department of Mathematics, University of Ottawa, 585 King Edward, Ottawa, ON K1N
6N5, Canada.}

\dedication{Dedicated to Mehrdad Shahshahani.}


\thanks{
Throughout the completion of this work, M.B 
was supported by Emmanuel Breuillard's ERC grant `GeTeMo',
K.M-K was partially supported by the DFG grant DI506/14-1, and H.S was supported by NSERC Discovery Grants RGPIN-2013-355464 and RGPIN-2018-04044. }

\begin{abstract}
Given a finite group $\G$ and a field $K$, the {\it faithful dimension} of $\G$ over $K$ is defined to be the smallest integer $n$ such that $\G$ embeds into $\GL_n(K)$. 
We address the problem of determining 
the faithful dimension of a $p$-group of the form $\GG_q:=\exp(\g \otimes_\Z\F_q)$
associated to $\g_q:=\g \otimes_\Z\F_q$ in the Lazard correspondence, where 
$\g$ is  a nilpotent $\Z$-Lie algebra 
which is finitely generated as an abelian group.
We show that in general the faithful dimension of $\GG_p$ is a piecewise polynomial function of $p$ on a partition of primes into Frobenius sets. 
Furthermore, we prove that for $p$ sufficiently large, there exists a partition of $\mathbb N$ by sets from the Boolean algebra generated by arithmetic progressions, such on each part the faithful dimension of $\GG_q$ for $q:=p^f$ is equal to $f g(p^f)$ for a polynomial $g(T)$. 
We show that for many naturally arising $p$-groups, including a vast class of groups defined by partial orders, the faithful dimension is given  by a single formula of the latter form. The arguments rely on various tools from number theory, model theory, combinatorics and Lie theory. 
\end{abstract}

\maketitle


\section{Introduction}
Let $\G$ be a finite group and let $K$ be a field. The {\it faithful dimension} of $\G$ 
  over $K$, denoted by $m_{\mathrm{faithful},K}(\G)$, is defined to be the smallest possible dimension of a faithful $K$-representation of $\G$. 
 The question of computing or estimating  $m_{\mathrm{faithful},K}(\G)$ has found many applications. 
For instance, it is intimately connected to computing the essential dimension $\mathrm{ed}_K(\G)$ of $\G$, defined by Buhler and 
Reichstein 
\cite{Reichstein},
which is  the smallest dimension of a linearizable $\G$-variety with a faithful $\G$-action. 
 It is known \cite[Proposition 4.15]{BerFav}
that $\mathrm{ed}_K(\G)\leq m_{\mathrm{faithful},K}(\G)$ for every finite group $\G$.  
Karpenko and Merkurjev~\cite{Merkurjev} proved that 
if $\G$ is a $p$-group and  $K$ contains a primitive $p$th root of unity, then $\mathrm{ed}_K(\G)=m_{\mathrm{faithful},K}(\G)$. For further details the reader may wish to consult~\cite{MerkurjevI}.

Note that by a result of Brauer, every complex representation of a $p$-group $\G$ is defined over $\mathbb Q(\zeta)$, where $\zeta$ is a primitive $|\G|$-th root of unity.
 This implies that
 $m_{\mathrm{faithful},K}(\G)
=m_{\mathrm{faithful},\C}(\G)$
whenever $K\supseteq \mathbb Q(\zeta)$.
Therefore we will only consider complex representations and use the shorthand $\mf(\G)$ instead of $m_{\mathrm{faithful},\C}(\G)$.

This work is a continuation of~\cite{BMS} in which the faithful dimension of a large class of $p$-groups was studied. Let us start by recalling some of the results from \cite{BMS}.
Let $F$ be a non-Archimedean local field with a discrete valuation $\nu$. 
We will denote the  
ring of integers 
 of $F$ 
by $\mathcal{O}$, the unique maximal ideal of $\mathcal{O}$ by $\mathfrak{p}$, and the  
residue field $\mathcal{O}/\mathfrak{p}$ 
by $\F_q$, the finite field of order $q:=p^f$, where  $f$ is the {\it absolute inertia degree} of $F$.
The number $e=\nu(p)$ is called the {\it absolute ramification index} of $F$.

For a (commutative and unital) ring $\R$ and 
an integer $k \ge 1$, the $k$th {\it Heisenberg group} with entries in $\R$, denoted by $\Hei_{2k+1}(\R)$, consists of $(k+2)\times (k+2)$ matrices of the form $I_{k+2}+A$, where $A$ is strictly upper triangular and all of its entries other than those on the first row and the last column are zero. Similarly, $\U_k(\R)$ denotes the subgroup of unitriangular matrices in $\GL_k(\R)$, so that $H_{2k+1}(\R) \subseteq U_{k+2}(\R)$.
In~\cite[Theorem 1.1]{BMS} 
we proved
that
\begin{equation}\label{Heis}
\mf\left(\Hei_{2k+1}(\OO{n})\right)=\sum_{i=0}^{ \xi-1 } fq^{k(n-i)},
\end{equation}
where $\xi=\min \left\{ e,n \right\}$.
Also, when
$\mathrm{char}(\mathcal{O}/\mathfrak{p})\neq 2$,
in~\cite[Theorem 1.2]{BMS} 
we showed
that

\begin{equation*}
\mf(\G)=\mf(\Hei_{2k+1}(\OO{n})), \qquad\,
\end{equation*}
for any subgroup $\G$ of $\U_{k+2}(\OO{n})$
that contains
$\Hei_{2k+1}(\OO{n})$. 
In particular for
$F=\F_p((T))$, where $p\geq 3$, we obtained
\begin{equation*}
\mf(\U_k(\F_p[[T]]/(T^n)))=\sum_{i=0}^{n-1} p^{(k-2)(n-i)}\quad\text{for all } k\geq 3.
\end{equation*}
The latter statement implies that if $p \ge 3$ then
\begin{equation}\label{U-al}
\mf(\U_{k}(\F_p))=p^{k-2} \quad \text{ for all }k\geq 3.
\end{equation} 
The right hand side of~\eqref{U-al} 
is a polynomial in $p$.  Note that  
$\U_{k}(\F_p)=\exp(\mathfrak{u}_k\otimes\F_p)$, where
$\mathfrak{u}_k$ is 
the Lie algebra of strictly upper triangular matrices with entries in $\Z$ 
(for the definition of 
the exponential map in this context, see Section~\ref{Main-Section}). Equation~\eqref{U-al} suggests the following problem. 

\begin{problem}[(Polynomiality problem)] 
Among all nilpotent 
$\Z$-Lie algebras $\g$ which are finitely generated as abelian groups,
characterize those for which there exists a polynomial $g(T)$, only depending on $\g$, such that 
\begin{equation*}
\mf(\exp(\g\otimes_\Z\F_p))=g(p),
\end{equation*}  
for all sufficiently large primes $p$.
\end{problem}

This problem is the central guiding principle of this work. Before stating  
our results, let us mention that the methods applied 
in~\cite{BMS}
were based on a suitably adapted version of the Stone-von Neumann theory, whose application is mostly limited to groups of nilpotency class $2$. In this paper, we will replace the Stone-von Neumann theory by {\it Kirillov's orbit method} for finite $p$-groups. The orbit method was initially introduced by Kirillov~\cite{Kirillov} to study unitary representations of nilpotent Lie groups. This machinery was later adapted 
to other classes of groups, such as $p$-adic analytic groups, finitely generated nilpotent groups, and finite $p$-groups (see~\cite{HoweI, HoweII} and~\cite[Proposition 1]{Kazhdan}).
Jaikin-Zapirain~\cite{Jaikin} used this machinery to study the representation zeta functions 
 of compact $p$-adic analytic groups. 
 These zeta functions have also been studied by Avni, Klopsch, Onn and Voll~\cite{Avni}.
We refer the reader to these papers and references therein for more details.

Our approach to Problem 1.1 relies heavily on the notion of the {\it commutator matrix} associated to a nilpotent $\Z$-Lie algebra. 
Since its introduction in the work of Grunewald and Segal~ 
\cite{Grunwald-Segal} as a tool for classification of torsion-free finitely generated nilpotent groups, the notion of the commutator matrix has been used in  studying a large variety of problems related to finite and infinite groups. Voll~\cite{Voll04,Voll05} has used commutator matrices  in his works on normal subgroup lattices of nilpotent groups. Stasinski and Voll~\cite{Stasinski-Voll} also employed them to study
the representation growth of infinite groups. In addition, O'Brien and Voll~\cite{Voll} used commutator matrices for counting conjugacy classes and characters of certain finite $p$-groups. 
In our work, 
we relate the faithful dimension of  finite $p$-groups to the question of existence of sufficiently many points in general position on rank varieties associated with the commutator matrices that were considered by O'Brien and Voll. 

\section{Main results}\label{Main-Section} Before we state our results we will need to set some notation. 
Let $\g$ be a Lie algebra over a commutative ring $\R$.
For $x\in \g$, the map defined by $y\mapsto [x,y]$  is denoted by $\ad_x$. 
Let $(\g^l)_{l\geq 1}^{}$ denote the 
{\it descending central series} of $\g$. 
In other words we set $\g^1:=\g$, and we define $\g^{l+1}$ for $l\geq 1$ inductively, as the $\R$-submodule of $\g$ generated by commutators of the form $[x,y]$, where $x\in\g$ and  $y\in\g^l$. The commutator subalgebra of $\g $ will be denoted by $\g'$. Note that $\g'=\g^2$. 
 The Lie algebra $\g$ is said to be {\it nilpotent} if $\g^{c+1}=0$ for some $c\in\mathbb{N}$. If $c$ is the smallest integer with this property, then $\g$ 
 is said to be \emph{$c$-step nilpotent} or \emph{nilpotent of class $c$}.

Suppose now that $\g$ is a finite $\Z$-Lie algebra
 whose cardinality is a power of $p$, and assume that $\g$ is nilpotent of class $c<p$. One may define a group operation on $\g$ by the Campbell-Baker-Hausdorff formula: For all $x,y\in\g$ we define the group multiplication by
\begin{align*}
x*y:=
\sum_{n>0}\frac{(-1)^{n+1}}{n}\sum_{\substack{(a_1,b_1),\dots, (a_n,b_n)\\ a_j+b_j\geq 1}}\frac{(\sum_{1\leq i\leq n} a_i+b_i)^{-1}}{a_1!b_1!\cdots a_n!b_n!}(\ad_x)^{a_1}(\ad_y)^{b_1}\dots (\ad_x)^{a_n}(\ad_y)^{b_n-1}(y),\notag
\end{align*}
where if $b_n=0$ then the last term $(\ad_y)^{b_n-1}(y)$ is dropped. 
Plainly, if $b_n>1$, or if $b_n=0$ and $a_n>1$, then the corresponding summand vanishes. Note that the above sum is finite, because $\g$ is nilpotent. The group defined in this way is denoted by $\exp(\g)$.
For instance, when $\g$ is 2-step nilpotent and $p\geq 3$, 
the group multiplication of $\exp(\g)$ takes the simple form
$$
x\ast y=x+y+\frac{1}{2}[x,y],$$ 
and when $\g$ is 
3-step nilpotent and $p\geq 5$, we obtain
$$x\ast y=x + y +\frac{1}{2}[x,y]+\frac{1}{12}[x,[x,y]]+\frac{1}{12}[y,[y,x]].$$
Similar formulas can be written for any given nilpotency class when $p$ is large enough. The group $\exp(\g)$ defined above is a $p$-group of nilpotency class $c$. In fact Lazard proved~\cite[Chapter 9]{Khukhro} that every $p$-group $\G$ of nilpotency class $c<p$ arises in this way from a unique Lie algebra $\g:=\mathrm{Lie}(\G)$.

From now on, for a $c$-step nilpotent $\Z$-Lie algebra 
$\g$ 
which is finitely generated as an abelian group,
and for $q:=p^f$ 
with
 $p>c$, we set 
\[
\g_q:=\g \otimes_\Z\F_q\quad\text{ and }\quad
\GG_q:=\exp(\g_q),
\] where 
 $\F_q$ is the finite field with $q$ elements.

\subsection{A palette of possibilities} To illustrate the range of possibilities that can arise, we will start this section with three examples, and then state our  main results. We will elaborate on these examples
 in Section~\ref{Nilpotent-Section}.

\begin{example}[(Elliptic curve)]\label{cubic-curve-example} Let $a$ be a non-zero integer. Consider the $\Z$-Lie algebra $\g_a$, introduced by Boston and Isaacs~\cite[Section 3]{Boston}, which is spanned as a free $\Z$-module by  $\{v_1,\dots,v_9\}$, subject to the relations 
$$
[v_1,v_4]=[v_2,v_5]=[v_3,v_6]=v_7, \ \  [v_1,v_5]=[v_2,v_6]=v_8, \ \  [v_1,v_6]=av_9, \ \  [v_2,v_4]=[v_3,v_4]=v_9.
$$
All other brackets $[v_i,v_j]$ with $i<j$ vanish.  It will be shown 
in Section \ref{cubic-curve-example-sol}
that if $p$ is a sufficiently large prime ($p>1800$  will suffice) and $p$ does not divide $a$, then 
\begin{equation*}
\mf(\exp(\g_a\otimes_\Z\F_p))=3p^2.
\end{equation*}
As we will see in the proof, the uniformity in $p$ is related to the fact that for such values of $p$ the cubic curve
$
Y^2=4aX^3+X^2-4X
$
 has a non-zero rational point over $\F_p$. Note that in this example, aside from a finite set of primes, the value of $\mf(\exp(\g_a\otimes_\Z\F_p))$ is given by one polynomial in $p$. 
\end{example}

\begin{example}[(Binary quadratic form)]\label{Binary quadratic form} Consider the $\Z$-Lie algebra $\g$ spanned as a free $\Z$-module by $\{v_1,\dots,v_6\}$ subject to the  relations:
\begin{equation*}
[v_1,v_2]=[v_3,v_4]=v_5,\quad [v_1,v_4]=[v_2,v_3]=v_6, 
\end{equation*} 
where all other commutators $[v_i,v_j]$ with $i<j$ are defined to be $0$.  Then  
in Section~\ref{Binary quadratic form-sol}
we will show  that for odd primes $p$, the value of  $\mf(\GG_p)$ is given by two different polynomials along two arithmetic progressions, namely, 
\begin{equation*}
\mf(\GG_p)=\begin{cases}
2p & \text{if }\,p\equiv 1 \pmod{4};\\
2p^2 & \text{if }\, p\equiv 3 \pmod{4}.
\end{cases}
\end{equation*}  
Put  more formally, set
$$\mathscr{P}_1:=\{p\geq 3: p\equiv 1\pmod{4}\}\quad
\text{ and }\quad \mathscr{P}_2:=\{p\geq 3: p\equiv 3\pmod{4}\}.$$
Also, set $g_1(T):=2T$ and $g_2(T):=2T^2$. 
Then $
\mf(\GG_p)=g_i(p)
$
for all $p\in\mathscr{P}_i$, $i=1,2$.
\end{example}

As the next example shows, even this is not the end of the story.
  
\begin{example}[(Binary cubic form)]\label{Binary cubic form}
Consider the $\Z$-Lie algebra $\g$ spanned 
as a free $\Z$-module by $\{v_1,\dots,v_8\}$ with the following relations:
\begin{equation*}
[v_1,v_4]=[v_2,v_5]=[v_3,v_6]=v_7,\quad [v_1,v_5]=[v_2,v_6]=[v_3,v_5]=v_8,\quad [v_3,v_4]=-v_8.
\end{equation*} 
All other commutators $[v_i,v_j]$ with $i<j$ vanish.  
Let $p$ be an odd prime. 
In Section \ref{Binary cubic form-sol} we will show that  
\begin{equation*}
\mf(\GG_p)=\begin{cases}
p^2+p^3 & \text{if }\,\left(\frac{p}{23}\right)=-1;\\
2p^3 & \text{if }\,p\,\, \text{is represented by the form\,}  2x^2+xy+3y^2; \\
2p^2 & \text{if }\,p\,\, \text{is represented by the form\,} x^2+xy+6y^2\quad \text{or}\quad p=23,
\end{cases}
\end{equation*}
where $(\frac{\cdot}{p})$ is the Legendre symbol. 
The conditions defining this function split the set of prime numbers $p\geq 3$ into disjoint sets $\mathscr{P}_1, \mathscr{P}_2$ and $\mathscr{P}_3$. On each one of these sets, one of the polynomials $g_1(T)=T^2+T^3$, $g_2(T)=2T^3$ and $g_3(T)=2T^2$ is applicable. It is worth mentioning that 
by Gauss genus theory,  
the sets $ \mathscr{P}_2$ and $\mathscr{P}_3$ 
are {\it not} unions of arithmetic progressions
(for example see \cite{Kusaba}).
\end{example}

\begin{example}[(Lee's Lie algebra)]\label{LeeLie}
Consider the $\Z$-Lie algebra $\g$ spanned 
as a free $\Z$-module by $\{v_1,\dots,v_8\}$ with the following relations:
\begin{equation*}
[v_1,v_4]=[v_2,v_5]=v_6,\quad2[v_1,v_5]=[v_3,v_4]=2v_7,\quad[v_2,v_4]=[v_3,v_5]=v_8.\end{equation*} 
All other commutators $[v_i,v_j]$ with
$i<j$ vanish. 
This Lie algebra was  defined in \cite{SLee}. 
Let $p$ be an odd prime. 
In Section \ref{LeeLie-sol} we will show that 
\begin{equation*}
\mf(\GG_p)=\begin{cases}
p+2p^2 & \text{if }\,p\equiv 2\pmod{3}\text{ or }p=3;\\
3p & \text{if }\,p\equiv 1\pmod{3}\text{ and }p\,\, \text{is represented by the form\,}  x^2+27y^2; \\
3p^2 & \text{if }\,p\equiv 1\pmod{3}\text{ and }p\,\, \text{is not represented by the form\,}  x^2+27y^2.
\end{cases}
\end{equation*}
We remark that in \cite{SLee}, the author computes   the order of the automorphism group of the Lie algebra $\g_p$. 
The formula obtained in \cite{SLee} is given by different polynomials depending on  the splitting of  $\lambda^3-2$ 
in $\F_p$, and therefore its cases are  parallel to the ones that appear above.
\end{example}

The important point to note here is that in each one of the above examples, the set of primes can be decomposed into finitely many {\it arithmetically defined} sets, such that the value of $\mf(\GG_p)$ on each 
one of these sets is given by a polynomial. 
Let us explain this in more detail. For a polynomial $g(T)\in\Z[T]$, denote by $\V_{g}$ the set of primes $p$ for which the congruence $g(T)\equiv 0 \pmod{p}$ has a solution. 
We will call $\V_{g}$ an {\it elementary Frobenius set}. Let $\mathbb{P}$ denote the set of prime numbers. 
By a {\it Frobenius set} we mean 
an element of the Boolean algebra inside the power set of $\mathbb{P}$ that is  
generated by elementary Frobenius sets. In other words, a Frobenius set is a finite union of sets
of the form
$$
\V_{g_1}\cap \dots \cap\V_{g_k}\cap \V_{g_{k+1}}^c\cap \dots\cap \V_{g_l}^c.
$$
We remark that every  Frobenius set is a \emph{Frobenian set}, as defined by Serre~\cite[Section 3.3.1]{SerreNX}, but the converse does not hold.
For more details about the connection between Frobenius and Frobenian sets, see~\cite{Lagarias}.

For $p>2$ the equation $x^2+1=0$ has a solution in $\F_p$ if and only if $p \equiv 1 \pmod{4}$. This shows that the sets appearing in Example~\ref{Binary quadratic form} are Frobenius sets. 
One can see that the sets $\mathscr{P}_i$ in Example~\ref{Binary cubic form} are Frobenius sets as follows.
First, using the quadratic reciprocity law one can easily verify that 
the set $\mathscr{P}_1$ consists of those primes $p\geq 3$ for which the equation $x^2+23$ has no solution in $\F_p$. The other two parts are based on the less trivial fact that $p\geq 3$ can be represented by the quadratic form $a^2+ab+6b^2$ (respectively,  $2a^2+ab+3b^2$) if and only if $p\not\in\mathscr{P}_1$ and $x^3-x-1$ has a solution (respectively, no solution) in $\F_p$. Consequently, each $\mathscr{P}_i$ is a Frobenius set.

Let $\mathscr{P}$ be any set of primes. The {\it Dirichlet density} of $\mathscr{P}$ is defined by
\begin{equation*}
d(\mathscr{P}):=\lim_{s\to 1^+}\frac{\sum_{p\in \mathscr{P}}p^{-s}}{\sum_{p\in \mathbb{P}}p^{-s}},
\end{equation*}
if the limit exists. 
The Chebotarev density theorem implies that any infinite Frobenius set has a positive Dirichlet density. It follows from Dirichlet's theorem on primes in arithmetic progressions that the Dirichlet density of the sets $\mathscr{P}_i$ in  Example~\ref{Binary quadratic form} is positive. For the sets  $\mathscr{P}_i$  in Example~\ref{Binary cubic form},
positivity of the Dirichlet density 
can be proved by
class field theory. For more details we refer the reader to~\cite[Theorem 9.12]{Cox}.
\medskip

 With this preparation, we are now ready to state our first result.
\begin{theorem}\label{rationality-Ax}
Let $\g$ be a nilpotent $\Z$-Lie algebra of nilpotency class $c$ which is finitely generated as an abelian group. 
Then there exist a partition $\mathscr{P}_1, \dots, \mathscr{P}_r$ of the set of prime numbers larger than $c$ into Frobenius sets, and  polynomials $g_1(T),\dots, g_r(T)$ with non-negative integer coefficients, depending only on $\g$,  such that 
\[
\mf(\GG_p)=g_i(p)\ \text{  for all } p\in\mathscr{P}_i,
\]
where
$ 1 \le i \le r$. 
\end{theorem}

The proof of this theorem relies on a theorem of Ax~\cite{Ax} from model theory (see also van den Dries~\cite{vandenDries}), coupled with a parameterization of irreducible representations provided by the Kirillov machinery.  
As we shall see, 
studying $\mf(\GG_p)$  
leads to questions related to the existence of rational points over finite fields of certain determinantal varieties associated with the commutator matrix.
We remark that our proof is effective
and 
we can use its method to describe the associated Frobenius sets and polynomials. We will demonstrate this by a detailed analysis of the above examples after we give the proof of 
Theorem \ref{rationality-Ax}.

One can also consider 
$\mf(\GG_q)$ for $q:=p^f$
when the prime $p$ is fixed and $f$ varies. 
Let $\mathscr{P}$ be one of the Frobenius sets 
in Theorem~\ref{rationality-Ax} with the
associated polynomial $g(T)\in\Z[T]$.
The following example shows that it is not necessarily true that 
$$
\mf(\GG_q)=fg(q)\text{ for all }f\geq 1.
$$

\begin{example}
\label{ex:verticalprime}
 Take $\g$ as in Example~\ref{Binary quadratic form} and set $\mathscr{P}:=\mathscr{P}_2$, where 
$\mathscr{P}_2$ 
is as in the same example.   
 Recall that 
$$\mf(\GG_p)=2p^2\ 
\text{ for all $p\in\mathscr{P}$}.
$$
However, 
in Section \ref{Binary quadratic form-sol} we shall demonstrate that 
$$
\mf(\GG_q)=
\begin{cases}
2fq & \text{$f$ is even;}\\
2fq^2 & \text{$f$ is odd.}
\end{cases}
$$
\end{example}
\medskip
Nevertheless, we prove the following theorem which shows that in general, the behaviour of $\mf(\GG_q)$ for $q:=p^f$ with $p$ fixed is
 very similar to the above example. 
\begin{theorem}\label{APsets}
Let $\g$ be a nilpotent $\Z$-Lie algebra of nilpotency class $c$ which is finitely generated as an abelian group. Fix a prime $p>C$, where 
$C$ is the constant given in \eqref{eq:bigconstantnt}.
Then there exist a partition $\mathscr{A}_1, \dots, \mathscr{A}_r$ of the set of natural numbers,   and  polynomials $g_1(T),\dots, g_r(T)$ with non-negative integer coefficients, depending on $p$ and $\g$, such that
\begin{itemize}
\item[1)] Each $\mathscr{A}_i$, $1\leq i\leq r$, is a union of a finite set and finitely many arithmetic progressions.
\item[2)]
 For all $1\leq i\leq r$, if $q=p^f$ where $f\in\mathscr{A}_i$, then
$\mf(\GG_q)= fg_i(q)$.
\end{itemize}

\end{theorem} 
The proof of Theorem~\ref{APsets} uses Dwork's theorem on rationality of zeta functions of varieties and the Skolem-Mahler-Lech theorem.

\begin{remark}It would be interesting to obtain a uniform generalization of Theorem~\ref{rationality-Ax}
and Theorem~\ref{APsets}. 
\end{remark}

At this point two remarks are in order. On the one hand, Theorem~\ref{rationality-Ax}
and Theorem~\ref{APsets} set an upper limit on how complicated the value of $\mf(\GG_q)$ as a function of $p$ and $f$ can be. It would be desirable to know  to what extent the collection of functions appearing in Theorem \ref{rationality-Ax} can be realized as $\mf(\GG_q)$ for some Lie algebra $\g$. On the other hand, one may still hope that at least for a large class of naturally arising Lie algebras $\g$, the function $\mf(\GG_q)$ is given by a {\it single} polynomial. 
In the next section we address these two  problems.
 
\subsection{Pattern groups}\label{Root-Groups-Sec}
Let us begin this section by introducing a large class of nilpotent groups which  can be viewed as a generalization of the $n$th Heisenberg group $\Hei_{2n+1}(\F_q)$ defined in the introduction.

Set $[n]:=\{1, 2,\dots , n\}$, and let $([n], \prec)$ be a partially ordered set. Without loss of generality, we can assume that if $i\prec j$ then $i < j$. To this partial order we 
assign the {\it pattern Lie algebra}
$$\g_\prec:=\Span_{\Z}\{\e_{ij}: i\prec j\}\subseteq \mathfrak{gl}(n,\Z),$$  
where $\e_{ij}$ denotes the $n\times n$  matrix whose unique nonzero entry is a  $1$ in the $(i,j)$ position.

Note that $\g_\prec$ is nilpotent since 
the commutator relation
$$
[\e_{ij},\e_{kl}]=\delta_{jk}\e_{il}-\delta_{li}\e_{kj}
$$
ensures that $\g_\prec$ is  a subalgebra of the Lie algebra $ \mathfrak{u}_n$ of strictly upper-triangular $n$ by $n$ matrices. For instance, the $(2n+1)$-dimensional Heisenberg Lie algebra corresponds to the partial order
\begin{equation}\label{Hei-poset}
1 \prec 2,3, \dots, n+1 \prec n+2.
\end{equation}
For all $i\prec j$ define 
$$
\alpha(i,j):=\#\{k\in [n]: i\prec k\prec j\}.
$$
Moreover the {\it length} of $([n],\prec)$, denoted by $\lambda_\prec$, is defined to be the maximum value of 
$r$ such that there exists a chain $i_0 \prec \cdots \prec i_r$ in $([n], \prec)$. For instance the length of the partial order given in~\eqref{Hei-poset} is equal to $2$. We remark that  $\lambda_\prec$ is equal to the nilpotency class of $\g_\prec$ 
(see Lemma \ref{generator-comm}).
\begin{definition}
\label{dfn:exxtr}
An ordered pair $(i,j)$  is called an {\it extreme pair} 
if $i$ 
is a minimal element in $([n], \prec)$, 
$j$ is a maximal element in $([n], \prec)$, and $i\prec j$. The set of extreme pairs will be denoted by $\Iex$. 
\end{definition}

We can now state our theorem on the faithful dimension of pattern groups.
\begin{theorem}\label{partial-order-Lie}
Let $\prec$ be a partial order of 
length 
$\lambda_\prec$ on the set $[n]$,
and let $q:=p^f$ where  
$p>\lambda_\prec$.
Then
\begin{equation*}
\mf(\exp(\g_\prec\otimes_\Z\F_q))=\sum_{_{(i,j)\in \Iex}} fq^{\alpha(i,j)}.
\end{equation*} 
\end{theorem}

Theorem~\ref{partial-order-Lie} generalizes~\eqref{U-al}, and its proof relies on Kirillov theory and more specifically on an explicit description of the size of coadjoint 
orbits in terms of combinatorial data of the partial order. Note that by Theorem \ref{rationality-Ax}
and Theorem \ref{APsets}, 
\emph{a priori}
 there is no 
reason to expect the faithful dimension to be given by a {\it single} formula of the form $fg(p^f)$ for a polynomial $g(T)$. 
Theorem~\ref{partial-order-Lie} is proved by
 a detailed analysis of the size of coadjoint orbits and  a combinatorial lemma due to Rado and Horn.

\begin{question}
Let $F$ be a non-Archimedean local field with the ring of integers $\mathcal{O}$, the unique maximal ideal $\mathfrak{p}$ and the associated residue field $\F_q$. Is it true that  $\mf(\exp(\g_\prec\otimes_\Z\OO{n}))$ is given by the formula
\begin{equation}\label{qn:generalization}
\sum_{\ell=0}^{ \xi-1 }\sum_{(i,j)\in \Iex} fq^{(n-\ell)\alpha(i,j)},\quad \xi=\min\{n,e\},
\end{equation}
where $f$ is the absolute inertia degree and $e$ is the absolute ramification index of $F$? Formula~\eqref{qn:generalization} is suggested by~\eqref{Heis} and Theorem~\ref{partial-order-Lie}. 
\end{question}

An immediate consequence of Theorem~\ref{partial-order-Lie} is the following corollary.
\begin{corollary} For any non-zero polynomial $g(T)\in \Z[T]$ with non-negative coefficients, there exists a nilpotent $\Z$-Lie algebra $\g$ 
which is finitely generated as an abelian group
  such that
\begin{equation*}
\mf(\GG_q)=fg(q),
\end{equation*} 
when $p\geq\deg g(T)+2$ and $f\geq 1$. 
\end{corollary}


\subsection{Relatively free nilpotent groups}
In this section we will turn to free objects in certain categories of nilpotent Lie algebras. Let $\f_{n,c}:=\f_{n,c}(\Z)$ be the free nilpotent $\Z$-Lie algebra on $n$ generators  and of class $c$; it is defined to be the quotient algebra $\f_n/\f^{c+1}_{n}$, where $\f_n$ is the free $\Z$-Lie algebra on $n$ generators, and $\f^{c+1}_n$ denotes the $(c+1)$-th term in the lower central series of $\f_n$ starting with $\f_n^1=\f_n$. It is well-known that the rank of the quotient $\f_n^c/\f_n^{c+1}$ (as a $\Z$-module) is given by Witt's formula
\begin{equation*}
r_n(c):=\frac{1}{c}\sum_{d|c}\mu(d)n^{c/d},
\end{equation*}
where $\mu$ is the M\"{o}bius function. Using the orbit method 
one can prove that
\begin{equation*}
\mf(\exp(\f_{n,c}\otimes_\Z\F_q))\geq r_n(c)fq.
\end{equation*}
This lower bound is sharp for $\f_{n,2}$ and $\f_{n,3}$.

\begin{theorem}\label{thm:free} Let $n \ge 2$ and let $\F_q$ be the finite field with $q=p^f$ elements. Then 
\begin{itemize}
\item[1)]   $\displaystyle\mf(\exp(\f_{n,2}\otimes_\Z\F_q))=\frac{n^2-n}{2}fq$\, for $p\geq 3$. 
\item[2)] 
 $\displaystyle\mf(\exp(\f_{n,3}\otimes_\Z\F_q))=\frac{n^3-n}{3} fq$\,
for $p \ge 5$.
\end{itemize}
\end{theorem}
The proofs of these results involve explicit computations with Hall bases and rely on a subtle combinatorial optimization.

\begin{remark}
\label{thm:f2c}
For $2\leq c\leq 6$, the value of $\mf(\exp(\f_{2,c}\otimes_\Z\F_p))$
is given by Table \ref{Table1} below.

\begin{table}[h!]
\centering
\[
\begin{array}{|| c c c ||}
\hline
c &\text{condition} & \mf(\exp(\f_{2,c}\otimes_\Z\F_p)) \Tstrut\Bstrut\\
\hline
2 & p\geq 3 & p\Tstrut\Bstrut\\
3 & p\geq 5 & 2p\Tstrut\Bstrut\\
4 & p\geq 5 & 3p\Tstrut\Bstrut\\
5 & p\geq 7 & 2p^2+4p\Tstrut\Bstrut\\
6 & p\geq 7 & p^3+3p^2+5p\Tstrut\Bstrut\\
\hline
\end{array}
\]
\caption{}
\label{Table1}
\end{table}
In section~\ref{subsec-outlineTable}, we outline the computations that yield the values of 
$\mf(\exp(\f_{2,c}\otimes_\Z\F_p))$ in Table \ref{Table1}.
However, we are not able to obtain any general formula for 
$\mf(\exp(\f_{2,c}\otimes_\Z\F_p))$
in terms of $c$.

\end{remark}

For a Lie algebra $\g$, 
let $\left(D^k\g\right)_{k\geq 0}$ be the derived series of $\g$. Thus $D^0\g=\g$, 
$D\g=[\g, \g]$, and $D^{k+1}\g= [D^k \g, D^k \g]$ for $k\geq 1$. Note that $\g/D^2\g$ is the largest metabelian 
quotient of $\g$. When $\g= \f_{n,c}$, this quotient is called the free metabelian Lie algebra of class $c$ on $n$ generators and will be denoted by 
$\m_{n,c}$.

\begin{theorem}\label{thm:free-meta} 
Let $c \ge 2$.  Then  
$
\mf(\exp(\m_{2,c}\otimes_\Z\F_q))=(c-1)fq
$\,
for $q:=p^f$ with $p>c$.
\end{theorem}
 
In the course of the proof of Theorem
\ref{thm:free-meta} 
we will see
that computing the faithful dimension of $\exp(\m_{2,c}\otimes_\Z\F_q)$ is linked to {\it rational normal curves}, that is,  the image of the {\it Veronese map} given by 
$$
\nu_{c-2}: \mathbf{P}^1(\F_q)\to \mathbf{P}^{c-2}(\F_q),\quad  [X_0:X_1]\mapsto [X_0^{c-2}: X_0^{c-3}X_1:\dots:X_1^{c-2}].
$$    
This suggests that other tools (e.g. from the theory of
determinental varieties)
might be relevant in the more general situation.

\section{Preliminaries}\label{pre}
In this section we introduce some notation and prove a number of basic facts which will be used throughout this paper. In particular, we will explain the connection between faithful representations and the central characters of their irreducible components. We will also briefly recall the orbit method.

\subsection*{Notation}\label{Notation} 
Let $\G$ be a group with the identity element $\1$. The centre and the commutator subgroup of $\G$ will be denoted, respectively, by $\ZZ(\G)$ and $\G'=[\G,\G]$.
For an abelian $p$-group $\G$, we write 
$$
\Omega_1(\G):=\{g\in \G: g^p=\1\},
$$
which is a $\FZ$-vector space. The Pontryagin dual of an abelian group $A$, i.e. $\Hom(A, \C^{\ast})$, will be denoted by $\widehat{A}$. Evidently, when $A$ is an elementary abelian $p$-group,  $\widehat{A}$ has a canonical $\FZ$-vector space structure. 
We denote the cardinality of a set $S$ by $\#S$.

\subsection{Central characters of faithful representations of $p$-groups} 

Let $A$ be a finite abelian group. We denote the minimal number of generators of $A$ by $d(A)$.  
For an exact sequence   of finite abelian groups $
0\to A_1 \to A\to A_2\to 0$, the numbers $d(A)$, $d(A_1)$ and $d(A_2)$ satisfy the inequalities
\begin{equation}\label{d(A)}
\max\{d(A_i): i=1,2\}\leq d(A)\leq d(A_1)+d(A_2).
\end{equation} 
The number of invariant factors of $A$ will be denoted by $d'(A)$. It can be easily seen from elementary divisor theory that $d'(A)=d(A)$.  
Evidently  $\mf(A)\leq d'(A)$. Now for a given faithful representation $\rho: A\to \GL_m(\C)$, by decomposing $\rho$ into irreducible components and applying~\eqref{d(A)} we obtain $d(A)=d'(A)\leq \mf(A)$. 
This implies that  
$$\mf(A)=d(A)=d'(A).$$
In particular, we obtain the following lemma.
\begin{lemma}\label{Tensor-pabelian-Lemma} 
For a finite abelian $p$-group $A$, 
\begin{equation*}
d(A)=d'(A)=\mf(A)=\dim_{\FZ}(A\otimes_{\Z} \FZ)=\dim_{\FZ}(\Omega_1(A)).
\end{equation*}
\end{lemma}
Let $E$ be a finite elementary abelian $p$-group equipped with the canonical $\FZ$-vector space structure. Every one-dimensional representation $ \chi: E \to \C^{\ast}$ factors uniquely as $ \chi=  \epsilon\circ \chi_\circ$, where $ \chi_\circ \in \Hom(E, \FZ)$ and the embedding $ \epsilon:\FZ\to \C^*$ is defined by \[
\epsilon(x+p\Z)=\exp((2\pi i x)/p).
\] Hence the $\FZ$-linear map
\begin{equation}\label{iso-Omei}
\widehat{E}\rightarrow \Hom(E,\FZ)\ ,\  \chi\mapsto \chi_\circ,
\end{equation}
provides an isomorphism of $\FZ$-vector spaces between 
$\widehat{E}$ and $\Hom(E, \FZ)$. 
Now, let $\G$ be a finite $p$-group. Applying~\eqref{iso-Omei} to $\Omega_1(\ZZ(G) )$, we  obtain the $\FZ$-isomorphism 
\begin{equation}\label{iso-Ome}
\Hom(\Omega_1(\ZZ(\G)),\C^*)\rightarrow \Hom(\Omega_1(\ZZ(\G)),\FZ).
\end{equation}
Hereafter the $\FZ$-vector space $\mathrm{Hom}(\Omega_1(\ZZ(\G)),\C^*)$ will be denoted by $\widehat{\Omega}_1(\ZZ(\G))$.

\begin{remark}\label{fact=Remark}
Recall the standard fact that for a finite $p$-group $\G$, every non-trivial normal subgroup of $\G$ intersects  $\ZZ(\G)$ and hence $\Omega_1(\ZZ(\G))$ non-trivially. Consequently, a representation of $\G$ is faithful if and only if its restriction to $\Omega_1(\ZZ(\G))$ is faithful.
\end{remark}
 We recall the following simple lemma. 
\begin{lemma}\label{Linear-map} Let $L, L_1,\dots, L_n$ be linear functionals on a vector space $V$ 
with respective null spaces $N$, $N_1,\dots, N_n$. Then $L$ is a linear combination of 
$L_1,\dots, L_n$ if and only if $N$ contains the intersection $N_1\cap\dots\cap N_n$.
\end{lemma}
The following observation, due to Meyer and Reichstein~\cite{ReichsteinI}, will play a crucial role in computing the faithful dimension of $p$-groups.
\begin{lemma}\label{central-span-faith} Let $\G$ be a finite $p$-group and let $(\rho_i,V_i)_{1\leq i\leq n}$ be a family of irreducible representations of $\G$ with  central characters $\chi_i$. Assume that the set of characters 
\[
\left\{\restr{{\chi_i}}{\Omega_1(\ZZ(\G))}: 1\leq i\leq n\right\}
\] spans $\widehat{\Omega}_1(\ZZ(\G))$. Then 
$\bigoplus_{1\leq i\leq n}\rho_i$
 is a faithful representation of $\G$.  
\end{lemma}
\begin{proof}
Since the set $\left\{\restr{{\chi_i}}{\Omega_1(\ZZ(\G))}: 1\leq i\leq n\right\}$ spans $\widehat{\Omega}_{1}(\ZZ(\G))$, from  the $\FZ$-isomorphism \eqref{iso-Ome} 
and Lemma~\ref{Linear-map} 
we obtain \[
\bigcap_{i=1}^n\ker \restr{{\chi_i}}{\Omega_1(\ZZ(\G))}=\{\1\}.
\] Hence $\bigoplus_{1\leq i\leq n}\rho_i$ is a faithful representation of $\Omega_1(\ZZ(\G))$. Remark~\ref{fact=Remark} implies that $\bigoplus_{1\leq i\leq n}\rho_i$ is a faithful representation of $\G$. 
\end{proof}
\begin{lemma}\label{Meyer} Let $\G$ be a finite $p$-group and let $\rho$ be a faithful representation of $\G$ with the smallest possible dimension. Then $\rho$ decomposes as a direct sum of exactly $r:=d(\ZZ(\G))$ irreducible representations
\begin{equation*}
\rho=\rho_1\oplus\dots\oplus\rho_r.
\end{equation*}
Therefore the set of central characters $\left\{\restr{{\chi_i}}{\Omega_1(\ZZ(\G))}: 1\leq i\leq r\right\}$ is a basis of $\widehat{\Omega}_1(\ZZ(\G))$.
\end{lemma} 
\begin{proof}
Let $\rho=\bigoplus_{1\leq i\leq n}\rho_i$
be the decomposition of $\rho$, and let $\chi_i$, $1\leq i\leq n$, denote the central character of $\rho_i$.  Since $\rho$ is faithful and $r=d(\ZZ(\G))$, it follows that $n\geq r$. 
Furthermore, faithfulness of  $ \rho$ also implies  $\bigcap_{i=1}^n\ker\chi_i=\{\1\}$. Hence, from Lemma~\ref{Linear-map}, Lemma~\ref{central-span-faith}, and the minimality of $\dim(\rho)$ it follows that $n=r$ and also that 
the set \[
\left\{\restr{{\chi_i}}{\Omega_1(\ZZ(\G))}: 1\leq i\leq r\right\}
\] is a basis of $\widehat{\Omega}_1(\ZZ(\G))$.
\end{proof}

\subsection{Kirillov's orbit method} The orbit method was introduced by Kirillov~\cite{Kirillov} to study unitary representations of simply connected nilpotent Lie
groups. For such a group $\G$ with Lie algebra $\g$, this method provides an explicit bijection between the unitary dual $\widehat{\G}$ of $\G$, and the set $\Hom^\mathrm{cont}(\g,\C^*)/\G$ of orbits of
the induced action of $\G$ on $\Hom^\mathrm{cont}(\g,\C^*)$, called the coadjoint orbits. Since Kirillov's work, this method has been extended to study representations of nilpotent groups in other contexts. Relevant to this work is the version applicable to finite $p$-groups, which we now briefly explain. For more details we refer the reader to~\cite{Boyarchenko}. Let $\G$ be a $p$-group of nilpotency class $c<p$. 
By the Lazard correspondence, there exists a unique finite $\Z$-Lie algebra 
$\g:=\mathrm{Lie}(\G)$
of cardinality $|\G|$ and nilpotency class $c$  such that $\G\cong\exp(\g)$. Note that in the definition of $\exp(\g)$, the underlying set of the group is $\g$ and the multiplication law is defined by the Campbell-Baker-Hausdorff formula. Usually we  identify the underlying sets of the group $\G$ and the Lie algebra $\g$. A simple application of 
the Campbell-Baker-Hausdorff formula shows that in this identification the centre of $\g$ (as a Lie algebra) will be mapped onto the centre of $\G$ as a group. 

Consider now the coadjoint  action of $\G$ on $\widehat{\g}:=\Hom_\Z(\g,\C^*)$,
defined  by
$$\theta^x(y):=\theta\left(\sum_{n=0}^c \frac{\ad_x^n(y)}{n!}\right), $$
where $x,y\in \g$ and $\theta\in \widehat{\g}$. 
Note that since $p>c$, the sum is  well-defined. 

\begin{theorem}\label{Kirillov}
Assume that $p\geq 3$ and let $\G$ be a $p$-group of nilpotency class $c<p$. Furthermore, assume that $\g =\mathrm{Lie}(\G)$. Then there exists a bijection between $\G$-orbits $\Theta\subseteq \widehat{\g}$ and  irreducible representations $\rho_\Theta\in \widehat{\G}$ such that Kirillov's character formula holds:
\begin{equation*}
\chi_\Theta(x):=\chi_{\rho_\Theta}(x)=|\Theta|^{-1/2}\sum_{\theta\in\Theta}\theta(x).
\end{equation*}
\end{theorem}

\begin{proof}
See~\cite[Theorem 2.6]{Boyarchenko}. 
\end{proof}
\begin{remark}
For an extension of Kirillov's orbit method to the case $p=2$, see~\cite[Theorem 2.6]{Stasinski-Voll}
\end{remark}

Let $\Theta\subset\widehat{\g}$ be the orbit of $\theta_0\in \widehat{\g}$. Then from  Kirillov's character formula we see that the central character of $\rho_\Theta$ is $\restr{{\theta_0}}{\ZZ(\g)}$ and 
\begin{equation}\label{dim-Kirillov}
\dim(\rho_\Theta)=|\Theta|^{1/2}=[\g: \mathrm{Stab}_\G(\theta_0)]^{1/2}.
\end{equation}

\begin{proposition}\label{stab} The stabilizer of $\theta_0$ is given by
\begin{equation}\label{size-stab}
\mathrm{Stab}_\G(\theta_0)=\{x\in \g: \theta_0([x,y])=1\quad \forall y\in \g\}.
\end{equation}
\end{proposition}

\begin{proof}
One inclusion is obvious. For the inclusion $ \subseteq$ note that 
\begin{equation*}
\mathrm{Stab}_\G(\theta_0)=\left\{x\in \g: \theta_0^x(y)=\theta_0(y),\ \forall y\in \g\right\}=\left\{x\in \g: \theta_0\left(\sum_{n=0}^c \frac{\ad_x^n(y)}{n!}\right)=\theta_0(y),\ \forall y\in\g\right\}.
\end{equation*}
Fix $x\in \mathrm{Stab}_\G(\theta_0)$. 
Then
\begin{equation}\label{Lie-nil-Kirillov}
\theta_0\left(\sum_{n=1}^c \frac{\ad_x^n(y)}{n!}\right)=0\text{ for all $y\in \g$.}
\end{equation}
Choose an arbitrary element $y\in \g^{c-1}$. Since $\g^{c+1}=0$, it follows  from~\eqref{Lie-nil-Kirillov} that $\theta_0(\ad_x(y))=0$. Next choose an arbitrary $y\in\g^{c-2}$, and note that 
$$
\sum_{n=1}^c\frac{\ad_x^n(y)}{n!}=\ad_x(y)+\frac{\ad_x^2(y)}{2}.
$$
In light of the previous step,  $\theta_0(\ad_x(y))=0$. Continuing this process, the claim follows for all $y \in \g$. 
\end{proof}

Let us now illustrate the power of the orbit method by showing how it can be used to compute the faithful dimension of certain 2-step nilpotent groups. 

\begin{proposition}\label{prop3.8}
Let $\G$ be a 2-step nilpotent $p$-group, where $p\geq 3$. Assume that $\ZZ(G)$ is cyclic. 
Then
$\mf(\G)=\sqrt{[\G:\ZZ(\G)]}$. 
\end{proposition}
\begin{proof}
 Let $\rho$ be a faithful representation of $\G$ of minimal dimension. Since the centre of $\G$ is cyclic, it follows from Lemma~\ref{Meyer} that $\rho$ is irreducible. By the Lazard correspondence $\G=\exp(\g)$, where $\g$ is a finite Lie algebra of class two. It suffices to show that $\dim\rho=\sqrt{[\g:\ZZ(\g)]}$. Thanks to the orbit method, there exists $\theta_0\in \widehat{\g}$ such that 
$$
\chi_\rho(x)=\frac{1}{\sqrt{|\Theta|}}\sum_{\theta\in\Theta}\theta(x),
$$
where $\Theta$ is the $\G$-orbit of $\theta_0$. Therefore 
$$
\dim\rho=\sqrt{|\Theta|}=\sqrt{[\g:\mathrm{Stab}_\G(\theta_0)]}. 
$$
Recall that we identify $\ZZ(\G)$ with $\ZZ(\g)$. The restriction to $\ZZ(\g)$ of the central character of $\rho$ is $\theta_0$, and so $\theta_0: \ZZ(\g)\to \C^*$ is faithful. Now let $x\in \mathrm{Stab}_\G(\theta_0)$. Then 
$\theta_0([x,y])=1$ 
for all $y\in \g$,
so that  $[x,y]=0$.  Consequently,  $\mathrm{Stab}_\G(\theta_0)=\ZZ(\g)$. This completes the proof. 
\end{proof}

The class of $p$-groups covered by Proposition~\ref{prop3.8} includes all the extra special $p$-groups for $p\geq 3$; a $p$-group $\G$ is called extra special when its centre $\ZZ(\G)$ has $p$ elements and $\G/\ZZ(\G)$ is an elementary abelian $p$-group.  It is well-known that an extra special $p$-group has order $p^{2n+1}$ for some positive integer $n$.  Thus the faithful dimension of $\G$ is $p^n$.

\section{Faithful dimension of pattern groups}\label{Poset-Section} 
This section is devoted to the proof of Theorem~\ref{partial-order-Lie}. We start by recalling some notation that was defined in Section~\ref{Root-Groups-Sec}. Let $\prec$ be a partial order on the set $[n]:=\{1,\dots,n\}$. We can associate to $\prec$ the Lie algebra defined by 
$$\g:=\g_\prec=\Span_{\Z}\{\e_{ij}: i\prec j\}\subseteq \mathfrak{gl}(n,\Z).$$ 
Recall that the length of $([n],\prec)$, denoted by $\lambda_\prec$, is defined to be the maximum value of 
$r$ such that there exists a chain $i_0 \prec \cdots \prec i_r$ in $([n], \prec)$.
\begin{lemma}\label{generator-comm}
The nilpotency class of $\g$ is equal to the length of $([n],\prec)$.
\end{lemma}

\begin{proof}[Proof of Lemma~\ref{generator-comm}] 
First note that  
\begin{equation}\label{comm-rel}
[\e_{ij},\e_{kl}]=\delta_{jk}\e_{il}-\delta_{li}\e_{kj}
\text{ for $i\prec j$ and $k\prec l$.}
\end{equation}

From~\eqref{comm-rel} it follows that $[\e_{ij},\e_{kl}]= \e_{il}$ when $i\prec j=k\prec l$ and  $[\e_{ij},\e_{kl}]=-\e_{kj}$ when $k\prec l=i\prec j$. In other cases $[\e_{ij},\e_{kl}]=0$.
Given a chain $i_0 \prec \cdots \prec i_r$ in $([n], \prec)$, one can see that
\[ \e_{i_0, i_r}=[\e_{i_0,i_1}, [ \e_{i_1, i_2},  [\dots  ,  \e_{i_{r-1}, i_r}] ] \dots ]  \neq 0,\]
and hence the nilpotency class of $\g$ is at least $r$. Similarly, one can see that a non-zero commutator of length $r+1$ leads to 
a chain of length $r+1$, proving the claim. 
\end{proof}

Recall that $\Iex$ is the set of extreme pairs
(see Definition \ref{dfn:exxtr}).
\begin{lemma} The centre $\mathrm{Z}(\g)$ of $\g$ is spanned by 
$\{\e_{ij}\,:\,(i,j)\in \Iex\}$. 
\end{lemma} 
\begin{proof}
We first show that $\e_{ij}$ with $(i,j) \in \Iex$ is in the centre of $\g$.
 From~\eqref{comm-rel} it follows that  
$$
[\e_{ij},\e_{kl}]=0\text{ for all $k\prec l$,}
$$
since $i$ is minimal and $j$ is maximal.
Conversely, suppose $z=\sum_{i\prec j}x_{ij}\e_{ij}\in\ZZ(\g)$. 
We show that for each $i_1\prec j_1$, if  $x_{i_1j_1}\neq 0$ then $(i_1,j_1)\in\Iex$. Assume $i_1$ is not minimal, and pick  a minimal element $k\prec i_1$. 
Then~\eqref{comm-rel} implies that
$$
0=\left[z ,\e_{ki_1}\right]=-\sum_{i_1\prec j}x_{i_1j}\e_{kj},
$$
and thus $x_{i_1j_1}=0$, which is a contradiction. A similar argument shows that $j_1$ is maximal. 
\end{proof}

An additive character $\psi:\F_q\to \C^*$ is called {\it primitive} if the pairing 
\[
\F_q\times \F_q\to \C^*\ ,\ (x,y)\mapsto \psi(xy)
\] is non-degenerate.
We fix a primitive character $\psi$
by choosing $\iota: \F_p\to \C^*$ to be a faithful character and defining $\psi(x):=\iota(\Tr(x))$, where  
$\Tr:=\mathrm{Tr}_{\F_q/\F_p}: \F_q\to \F_p$ is the trace map. Using $\psi$ we can identify the Pontryagin dual of the additive group of $\F_q$ with $\F_q$.
It follows that all characters  of $\g_q$ are obtained by vectors $\bb=(b_{ij})\in\bigoplus_{i\prec j}\F_q$, via 
\begin{equation*}
\psi_{\bb}\left(\sum_{i\prec j} x_{ij}\e_{ij}\right):=\psi\left(\sum_{i\prec j}b_{ij}x_{ij}\right).
\end{equation*}
By Lemma \ref{generator-comm}, the Lie algebra
 $\g_q$  has nilpotency class 
$\lambda_{\prec}$. 
In the rest of this section we assume that $p>\lambda_\prec$. By the orbit method every irreducible representation of $\GG_q$ is constructed from the orbit of a character $\psi_\bb \in \widehat{\g}_q$ in the coadjoint action. We denote the irreducible representation obtained from $\psi_\bb$ by $\rho_\bb$.

\begin{proposition}\label{level-prop-Min} Let $\bb=(b_{ij})$ be an element of $\bigoplus_{i\prec j}\F_q$ and let $\rho_{\bb}$ be the irreducible representation of $\GG_q$ associated to the orbit of $\psi_{\bb}$.   
For all $(i_1,j_1)\in \Iex$, if $b_{i_1j_1}\neq 0$ then 
\begin{equation*}
 \dim\rho_{\bf b}\geq q^{\alpha(i_1,j_1)}.
\end{equation*}
\end{proposition}

\begin{proof} Set $I:=\{(i,j): i\prec j\}$. For $x=\sum_{i\prec j}x_{ij}\e_{ij}\in \mathrm{Stab}_{\GG_q}(\psi_{\bb})$ and $y=\sum_{i\prec j}y_{ij}\e_{ij}\in \g_q$, it follows from ~\eqref{comm-rel} that 
\begin{equation*}
\begin{split}
1=\psi_{\bb}([x,y])&=\psi_{\bb}\left(\sum_{i\prec j,\, k\prec l} x_{ij}y_{kl}[\e_{ij},\e_{kl}]\right)=\psi\left(\sum_{i\prec j}\sum_{i\prec k\prec j}b_{ij}(x_{ik}y_{kj}-x_{kj}y_{ik})\right)\\
&=\psi\left(\sum_{i\prec j}\left(\sum_{k\prec i}b_{kj}x_{ki}-\sum_{j\prec l}b_{il}x_{jl}\right)y_{ij}\right).
\end{split}
\end{equation*}
Since $y_{ij}\in \F_q$ is arbitrary and $\psi$ is a primitive character,  we obtain 
a system of linear equations  
\begin{equation}\label{Linear-eq1}
L_{ij}(x_{st}):=\sum_{k\prec i}b_{kj}x_{ki}-\sum_{j\prec l}b_{il}x_{jl}=0,
\end{equation}
which describes the stabilizer of $\psi_{\bb}$.
The equations in~\eqref{Linear-eq1} have coefficients in $\F_q$ and are indexed by pairs $i,j$ such that $i\prec j$.
We now consider only the linear forms $L_{i_1i}(x_{st})$ and $L_{jj_1}(x_{st})$, for $i_1\prec i\prec j_1$ and $i_1\prec j\prec j_1$. From these $i$ and $j$, we obtain $2\alpha(i_1,j_1)$ linear equations with coefficients in $\F_q$, as follows:
\begin{equation}\label{Linear-eq2}
\begin{split}
b_{i_1j_1}x_{ij_1}&=-\sum_{i\prec k\neq  j_1}b_{i_1k}x_{ik}, \qquad i_1\prec i\prec j_1;\\
b_{i_1j_1}x_{i_1j}&=-\sum_{i_1\neq l\prec j}b_{lj_1}x_{lj}, \qquad i_1\prec j\prec j_1.
\end{split}
\end{equation}
From $b_{i_1j_1}\neq 0$ it follows that $x_{i j_1}$ ($i_1\prec i\prec j_1$) and $x_{i_1j}$ ($i_1\prec j\prec j_1$) are dependent variables and thus, by noticing that each linear form has $\#I$ variables, the number of solutions of the equations~\eqref{Linear-eq2} is at most 
\begin{equation}\label{Stab}
q^{\# I-2\alpha(i_1,j_1)}.
\end{equation}
Thus the size of the stabilizer~\eqref{Linear-eq1} is at most~\eqref{Stab} and this gives the lower bound by~\eqref{dim-Kirillov}.
\end{proof}

\begin{lemma}\label{level-lemma} Let $b$ be a non-zero element of $\F_q$. Fix $(i,j)\in \Iex$ and define ${\bf b}=(b_{kl})_{k\prec l}$, where $b_{ij}=b$ and the other components are zero. Then the dimension of the irreducible representation $\rho_{\bf b}$ is $q^{\alpha(i,j)}$.
\end{lemma}

\begin{proof} Set $I:=\{(i,j): i\prec j\}$.
The proof of Proposition~\ref{level-prop-Min}, namely equation~\eqref{Linear-eq1}, shows that the stabilizer of  $\rho_{\bf b}$ is defined by the equations  
$bx_{ik}=0$ and $bx_{kj}=0$, 
where $i\prec k\prec j$.  These show that the stabilizer has cardinality $q^{\# I-2\alpha(i,j)}$, and therefore the dimension of $\rho_{\bb}$ is $q^{\alpha(i,j)}$ by~\eqref{dim-Kirillov}. 
\end{proof}

Using this we now construct a faithful representation of  $\GG_q$.

\begin{lemma}\label{cons-faith} The group $\GG_q$ has a faithful representation of dimension 
$$
\sum_{(i,j)\in \Iex} fq^{\alpha(i,j)}.
$$
\end{lemma}

\begin{proof} First note that $\ZZ(\GG_q)\cong \ZZ(\g_q)$ and so 
\begin{equation}\label{Z(G)}
\widehat{\Omega}_1(\ZZ(\GG_q))\cong\widehat{\Omega}_1(\ZZ(\g_q))=\bigoplus_{(i,j)\in \Iex} \widehat{\Omega}_1(\F_q)\cong\bigoplus_{(i,j)\in \Iex}\F_q.
\end{equation}
Let $\omega_1,\dots,\omega_f$ be a basis of $\F_q$ over $\F_p$. For $(i,j)\in \Iex$ and  $1\leq l\leq f$, define the vectors 
${\bf b}_{l}(i,j)\in \bigoplus_{s\prec t}\F_q$,
with the $(i,j)$ coordinate equal to $\omega_l$ and the other coordinates equal to $0$. Then the set 
$$
\{\bb_{l}(i,j): 1\leq l\leq f, \quad (i,j)\in \Iex\}
$$
is a basis of $\ZZ(\g_q)$ as an $\F_p$-vector space. It follows that the set
$$\{\psi_{\bb_l(i,j)}: 1\leq l\leq f,\quad (i,j)\in \Iex\}$$
  is a basis of $\widehat{\Omega}_1(\ZZ(\g_q))$ and thus of $\widehat{\Omega}_1(\ZZ(\GG_q))$ by~\eqref{Z(G)}.
Since $\psi_{\bb_{l}(i,j)}$ is the central character of $\rho_{\bb_l(i,j)}$, it follows from Lemma~\ref{central-span-faith} that the representation
$$\rho:=\bigoplus_{1\leq l\leq f}\bigoplus_{(i,j)\in \Iex}\rho_{\bb_l(i,j)}$$
is faithful.
 By Lemma~\ref{level-lemma} the dimension of $\rho_{\bb_l(i,j)}$ is equal to $q^{\alpha(i,j)}$ and hence
\begin{equation}\label{min-dim}
\dim\rho=\sum_{(i,j)\in \Iex} fq^{\alpha(i,j)}.
\end{equation}
This finishes the proof. 
\end{proof}
We are now ready to prove Theorem~\ref{partial-order-Lie}.

\subsection{Proof of Theorem~\ref{partial-order-Lie}} Write $m:= \# \Iex$ and $n:=fm$. As before, we identify 
$\widehat{\Omega}_1(\ZZ(\g_q))$ with $\bigoplus_{(i,j)\in \Iex}\F_q$ which has dimension $n$ as an $\F_p$-vector space.
 Let $\rho$ be a faithful representation of $\GG_q$ with the smallest possible dimension. We will show that the dimension of $\rho$ is bounded from below by the 
right hand side of~\eqref{min-dim}. 
Using Lemma~\ref{Meyer} we can decompose $\rho$ as a direct sum of $n$ irreducible representations, each of which obtained via the orbit method as described above. Hence, we can write
\begin{equation*}
\rho= \bigoplus_{k=1}^n \rho_{{\bf a}_k},
\end{equation*}
with vectors $\ba_k$ given by
$${\bf a}_{k}=\left(a_{st}(k)\right)_{s\prec t}\in\bigoplus_{s\prec t}\F_q.$$
Since 
the central character of $\rho_{\ba_k}$
is the restriction of $\psi_{\ba_k}$,  Lemma~\ref{Meyer} implies that the set $\{\psi_{\ba_k}: 1\leq k\leq n\}$ is a basis of $\widehat{\Omega}_1(\ZZ(\GG_q))$ and therefore the set
\begin{equation*}
\left\{ \left(a_{ij}(k)\right)_{(i,j)\in\Iex}: 1\leq k\leq n \right\}
\end{equation*} 
is a basis of the $\F_p$-vector space $\bigoplus_{(i,j)\in \Iex}\F_q$.
At this point we  need a combinatorial lemma whose proof relies on a theorem of Rado and Horn.

\begin{lemma}\label{blocks}
Let $V$ be an $m$-dimensional  $\F_q$-vector space. Suppose $S=\{v_1,\dots,v_{fm}\}$ is a basis of $V$ viewed as 
a vector space over the subfield $\F_p$. Then 
there exists a partition $S_1,\dots,S_f$
of $S$ into $f$ sets of size $m$ such that each $S_i$ is a basis of $V$ as an $\F_q$-vector space. 
\end{lemma}
We will use the following theorem of Rado and Horn~\cite{Horn}. The proof of this theorem itself is based on Hall's marriage theorem and ideas from 
matroid theory. We refer the reader to~\cite[Section 18]{Bollobas} for more details.

\begin{theorem}[(Rado-Horn)]\label{RH-Theorem} Let $V$ be a vector space over a field $E$ and let $\{v_i: 1\leq i\leq M\}$ be  a set of non-zero vectors in $V$. Then the following statements are equivalent: 
\begin{enumerate}
\item The set $\{1,\dots, M\}$ can be partitioned into sets $\{\mathcal{A}_j\}_{j=1}^k$ such that $\{v_i: i\in \mathcal{A}_j\}$ is a linearly independent set for all $j = 1 , 2 , \dots , k$.
\item For all non-empty subsets $J\subseteq \{ 1 , \dots, M \}$,
$$
\# J \leq k\dim_E \mathrm{Span}_{E}\{v_j: j\in J\}.
$$
\end{enumerate}
\end{theorem}
 
\begin{proof}[Proof of Lemma~\ref{blocks}] We apply Theorem~\ref{RH-Theorem} with $k=f$ to the set of vectors $S$. Consider
an arbitrary set  $J\subseteq \{1,\dots,mf\}$ and let $d=\dim_{\F_q}\mathrm{Span}_{\F_q}\{v_j: j\in J\}$. Then 
\[
\#\mathrm{Span}_{\F_q}\{v_j: j\in J\}=q^d=p^{fd}
.
\] 
Clearly $\mathrm{Span}_{\F_p}\{v_j: j\in J\}\subseteq\mathrm{Span}_{\F_q}\{v_j: j\in J\}$ and since $\{v_j: j\in J\}$ is linearly independent over $\F_p$, we obtain 
$$
p^{ \# J }=\#\mathrm{Span}_{\F_p}\{v_j: j\in J\}\leq \#\mathrm{Span}_{\F_q}\{v_j: j\in J\}=p^{fd}.
$$
It follows from the above inequality that
$$
\# J \leq f \dim_{\F_q}\mathrm{Span}_{\F_q}\{v_j: j\in J\}.
$$
Thus by the Rado-Horn theorem, the set $\{1,\dots,fm\}$ 
can be partitioned into $\mathcal{A}_1,\dots,\mathcal{A}_f$ such that each of the sets
$\{v_\ell: \ell\in \mathcal{A}_i\}$
 is linearly independent over $\F_q$. 
Note that 
$\# \mathcal{A}_i\leq m$ 
 since $\dim_{\F_q}V=m$. But the $\mathcal{A}_i$'s partition $\{1,\dots,mf\}$ and so $\mathcal{A}_i$ has the size $m$ which implies that $\{v_\ell: \ell\in \mathcal{A}_i\}$ is a basis of $V$ over $\F_q$. 
\end{proof}

We return to the proof of Theorem~\ref{partial-order-Lie}. Set $V=\bigoplus_{(i,j)\in \Iex}\F_q$, which is an $\F_p$-vector space of dimension $n=mf$. Also set $v_k=(a_{ij}(k))_{(i,j)\in\Iex}\in V$. Recall that $S=\{v_k: 1\leq k\leq n\}$ is a basis of $V$ as an $\F_p$-vector space and so by Lemma~\ref{blocks} there exist $f$ disjoint sets $S_1,\dots,S_f$, each of size $m$, such that each $S_\ell$ is a basis of $V$ as 
an $\F_q$-vector space.  
For $1\leq \ell\leq f$, let $A_\ell$ denote an $m\times m$ matrix whose rows are elements of $S_\ell$. Note that $A_\ell$ is invertible, since $S_\ell$ is a basis of $V$ as an $\F_q$-vector space. Using the Leibniz expansion of the determinant of $A$,  we can assume that up to a permutation of the rows, all of the diagonal entries of $A_\ell$ are non-zero. 
Thus Proposition~\ref{level-prop-Min} implies that 
$$
\sum_{(i,j)\in \Iex}q^{\alpha(i,j)}\leq \sum_{\ba_k\in S_\ell} \dim\rho_{\ba_k}
\ \text{ for  $1 \le \ell \le f$}
.
$$
Summing over all $\ell$, we obtain
$$
f\sum_{(i,j)\in \Iex}q^{\alpha(i,j)}\leq \dim\rho,
$$
which finishes the proof.
 
\section{The commutator matrix of nilpotent Lie algebras}\label{Nilpotent-Section} 
We now consider general nilpotent Lie algebras by rebuilding the argument presented in Section~\ref{Poset-Section}. Let $\g$ be a nilpotent $\Z$-Lie algebra  of nilpotency class $c$ 
which is finitely generated as an abelian group,
and let $\F_q$ be a finite field with $q=p^f$ elements. 
We set $\g_q:=\g\otimes_\Z\F_q$ throughout this section.
In order to apply the orbit method, we will also assume that $p>c$.
Existence of torsion elements in $\g$ and some of its quotients results in some technical difficulties which are addressed in what follows.

We call a subset $S$ of a finitely generated abelian group $\Gamma$ a \emph{semibasis} if it represents a basis over $\Z$ of the free abelian group 
$\Gamma/\Gamma_\mathrm{tor}$, where $\Gamma_\mathrm{tor}$ denotes the subgroup of torsion elements of $\Gamma$. Clearly $\#S=\rank_\Z\Gamma$. We define $\sfe(\Gamma)$ to be the largest prime divisor of the exponent of $\Gamma_{\mathrm{tor}}$.

\begin{remark}
\label{rmk:tensoringn}
Let $v_1,\ldots,v_d$ be $\Z$-linearly independent vectors in a finitely generated abelian group $\Gamma$ such that $\rank_\Z(\Gamma)=d$, and
let $M$ be the subgroup of $\Gamma$ generated by the $v_i$'s. Set $q:=p^f$, where $f$ is a positive integer and $p$ is a prime such that $p>\sfe(\Gamma/M)$.
Then the elements $v_1\otimes_\Z 1,\ldots,v_d\otimes_\Z 1$ form a basis of the $\F_q$-vector space
$\Gamma_q:=\Gamma\otimes_\Z\F_q$.
 \end{remark}

\begin{remark}
\label{rmk:rmk52}
For every prime $p$ we have $[\g,\g]_q=[\g_q,\g_q]$. 
The equality $\ZZ(\g_q)=\ZZ(\g)_q$ also holds for $p$ sufficiently large. An explicit lower bound for $p$ can be obtained as follows. Let $\vv_1,\ldots,\vv_n$ be a semibasis of $\g/\ZZ(\g)$, and let $\ww_1,\ldots,\ww_m$ be a semibasis of $[\g,\g]$. Then for $1\leq i<j\leq n$ we can write
$[\vv_i,\vv_j]=\sum_{k=1}^m\lambda_{ij}^k \ww_k+\mathbf{y}_{ij}$, where $\lambda_{ij}^k\in\Z$ and $\mathbf{y}_{ij}\in[\g,\g]_\mathrm{tor}$. Setting $x_{j+n(k-1),i}:=\lambda_{ij}^k$, we obtain an $mn\times n$ matrix $X:=[x_{a,b}]$. Now for every $n\times n$ submatrix $X'$ of $X$ we define $\mathsf{m}(X'):=\max\{p\,:\,p|\det(X')\}$, where $\mathsf{m}(X'):=1$ whenever $\det(X')=\pm 1$.
Further, set $\mathsf{m}(X):=\min_{X'}\{\mathsf{m}(X')\}$, where the minimum is taken over $n\times n$ submatrices of $X$. For the equality $\ZZ(\g_q)=\ZZ(\g)_q$ it is enough to assume that 
$p>C_1$, where 
$$C_1:=\max\{\mathsf{m}(X),\sfe([\g,\g]),\sfe(\g/[\g,\g])\}.$$
\end{remark} 
 
Now let $\ww_1,\ldots,\ww_{l_1}$ be a semibasis of $\ZZ(\g)\cap[\g ,\g]$. 
Let $\ww_{l_1+1},\ldots,\ww_{m}$ be
 elements of 
$\g $ which represent a semibasis of 
$[\g,\g]/(\ZZ(\g)\cap [\g,\g])$. Finally, let $\zz_1,\ldots,\zz_{l_2}$ be elements of $\g$ which represent a semibasis of $\ZZ(\g)/(\ZZ(\g)\cap [\g,\g])$.
It is straightforward to verify that the vectors 
$\{\ww_1,\ldots,\ww_m,\zz_1,\ldots,\zz_{l_2}\}$ are $\Z$-linearly independent.
Clearly the choice of these vectors implies that $\rank_\Z(\ZZ(\g)+[\g,\g])=l_2+m$. 
Let $M$ be the $\Z$-submodule of $\g$ generated by 
the $\ww_i$, $1\leq i\leq m$, and the $\zz_j$, $1\leq j\leq l_2$, and set 
$$C_2:=\sfe(\g /M).$$  From Remark \ref{rmk:tensoringn} it follows that if $p>\max\{C_1,C_2\}$, where $C_1$ is defined in  Remark \ref{rmk:rmk52},
then after tensoring with $\F_q$, these vectors form a basis of the $\F_q$-vector space 
$[\g_q,\g_q]+\ZZ(\g_q)$.

Now let $\vv'_1,\ldots,\vv'_n$ be elements of $\g$ which represent a semibasis of $\g/\ZZ(\g)$. 
For $1\leq i<j\leq n$, there exist integers $\eta_{ij}^k$, $1\leq k\leq m$, such that the elements 
\[
\vv'_{i,j}:=\left(
[\vv'_i, \vv'_j]-\sum_{k=l_1+1}^m\eta^k_{ij} \ww_k
\right)
\in[\g,\g] \]
are torsion modulo 
$[\g,\g]\cap\ZZ(\g)$.
Set $K$ 
equal to 
 the exponent of 
$\left(
[\g,\g]/[\g,\g]\cap\ZZ(\g)\right)_\mathrm{tor}$. It follows that $K\vv'_{i,j}\in[\g,\g]\cap\ZZ(\g)$ for every $1\leq i<j\leq n$. Now set
$\vv_i:=K\vv'_i$ for $1\leq i\leq n$. Then there exist integers $\lambda_{ij}^k$ such that
\begin{equation}\label{commutator}
[\vv_i, \vv_j]=\sum_{k=1}^m\lambda^k_{ij} \ww_k+\mathbf{x}_{ij}
\text{ for every $1\leq i<j\leq n$}, 
\end{equation}
where $\mathbf x_{ij}\in\big([\g,\g]\cap\ZZ(\g)\big)_\mathrm{tor}$.
We remark that $\lambda_{ij}^k=K^2\eta_{ij}^k$ for
$l_1+1\leq k\leq m$.

For each $1\leq i, j\leq n$, we define 
the linear forms
\begin{equation*}\label{Lambda}
\Lambda_{ij}(T_1,\dots,T_m):= \sum_{k=1}^m \lambda_{ij}^k T_k\in \Z[T_1,\dots,T_m]. 
\end{equation*}
It is clear that $\Lambda_{ii}=0$ and $\Lambda_{ij}=-\Lambda_{ji}$ for $1\leq i,j \le n$. The {\it commutator matrix of $\g$}
(relative to the chosen ordered basis) is the skew-symmetric matrix of linear forms defined by 
\begin{equation}\label{commutating-matrix}
F_\g(T_1,\dots, T_m):=[ \Lambda_{ij}(T_1, \dots, T_m)]_{1 \le i,j \le n }
\in \mathrm{M}_n(\Z[T_1,\dots,T_m]). 
\end{equation}
This matrix has  previously been used 
in several papers, such as those by
 Grunewald and Segal~\cite{Grunwald-Segal},  
Voll~\cite{Voll04,Voll05}, 
O'Brien and Voll
\cite{Voll},
Avni, Klopsch, Onn and Voll \cite{Avni},
and Stasinski and Voll \cite{Stasinski-Voll}.

  The following theorem expresses the faithful dimension of $\GG_q$ as the solution to a rank minimization problem. 
For the next theorem, we set
$$C_3 :=\sfe\left((\g/\ZZ(\g))_\mathrm{tor}^{}\right).$$
\begin{theorem}\label{Algorithm}
Let $\g$ be a nilpotent $\Z$-Lie algebra of nilpotency class $c$ which is finitely generated as an abelian group.  
If
$p>\max\{c,C_1,C_2,C_3\}$, then
\begin{equation*}
\mf(\GG_q)=\min\left\{ \sum_{\ell=1}^{l_1} fq^{\frac{\rank_{\F_q}(F_\g(x_{\ell 1},\dots,x_{\ell m}))}{2}}: \begin{pmatrix}
x_{11} & \dots & x_{1l_1}\\
\vdots  & \ddots & \vdots\\
 x_{l_1 1} & \dots & x_{l_1 l_1}
\end{pmatrix}\in\mathrm{GL}_{l_1}(\F_q)\right\}+fl_2,
\end{equation*} 
where $m:=\rank_\Z([\g,\g])$, 
$l_1:=\rank_\Z([\g,\g]\cap\ZZ(\g))$ and $l_2:=\rank_\Z(\ZZ(\g)/\ZZ(\g)\cap[\g,\g])$. 
\end{theorem}

\begin{remark}\label{reamr-assumption} 
We note that when $n=0$, the commutator matrix is the zero matrix and thus from Theorem~\ref{Algorithm} we obtain $\mf(\GG_q)=(l_1+l_2)f$. This formula can also be  obtained from Lemma~\ref{Tensor-pabelian-Lemma} since in this case for
$p$ as in Theorem
\ref{Algorithm} 
the group $\GG_q$ is abelian.
\end{remark}

\subsection{Proof of Theorem \ref{Algorithm}}

In this section we assume that $p$ is chosen as in 
Theorem \ref{Algorithm}.
By abuse of notation, we denote the images in $\g_q$ of the $\vv_i$'s, the $\ww_i$'s and the $\zz_i$'s that are chosen above by the same letters.
Let $\psi:\F_q\to \C^*$ be the primitive additive character defined in Section~\ref{Poset-Section}. Choose a basis $$\left\{\uu_1+(\ZZ(\g_q)+\g_q'),\dots, \uu_{l_3}+(\ZZ(\g_q)+\g_q')\right\}$$
of $\g_q/(\ZZ(\g_q)+\g_q')$.  
Since $p>C_2$,  the set
\begin{equation*}
\{\ww_1,\dots, \ww_{l_1}, \ww_{l_1+1},\dots, \ww_m,\zz_1,\dots, \zz_{l_2},\uu_1,\dots,\uu_{l_3}\}
\end{equation*}
is a basis of $\g_q$.
For 
\begin{equation}\label{vec-assigne}
\ba=(a_1,\dots,a_{l_1},a_{l_1+1},\dots,a_{m},b_{1},\dots,b_{l_2},c_1,\dots,c_{l_3})\in\F_q^{m+l_2+l_3},
\end{equation}
let $\psi_\ba\in\widehat{\g}_q$ be defined by
\begin{equation*}
\begin{split}
&\psi_\ba\left(\sum_{i=1}^{l_1}w_i\ww_i+\sum_{i=l_1+1}^{m}w_i\ww_i+\sum_{i=1}^{l_2}z_i\zz_i+\sum_{i=1}^{l_3}u_i\uu_i\right)
:=\\
&\qquad\qquad
\qquad\qquad\qquad\qquad
\psi\left(\sum_{i=1}^{l_1}a_iw_i+\sum_{i=l_1+1}^{m}a_iw_i+\sum_{i=1}^{l_2}b_iz_i+\sum_{i=1}^{l_3}c_iu_i\right).
\end{split}
\end{equation*}
The assignment $\ba\mapsto \psi_\ba$  identifies $\widehat{\g}_q$ with $\F_q^{m+l_2+l_3}$.
For $\ba$ as in~\eqref{vec-assigne}, we write
$\ba=(\ba',\ba'',\mathbf b,\mathbf c)$,
where $\ba'\in\F_q^{l_1}$, $\ba''\in\F_q^{m-\ell_1}$,  $\mathbf b\in\F_q^{l_2}$, and $\mathbf c\in\F_q^{l_3}$,
and define the projection maps
\[
\pr_1: \F_q^{m+l_2+l_3}\longrightarrow \F_q^{m+l_2}
\quad,\quad 
\pr_2: \F_q^{m+l_2+l_3} \longrightarrow \F_q^{l_1+l_2}\quad,\quad
\pr_3: \F_q^{m+l_2+l_3} \longrightarrow \F_q^{m},
\]
by 
$\pr_1(\ba)=(\ba',\ba'',\mathbf b)$, 
$\pr_2(\ba)=(\ba',\mathbf b)$, and 
$\pr_3(\ba)=(\ba',\ba'')$. 

In the rest of this section, we identify $\ZZ(\GG_q)$ with $\ZZ(\g_q)$. 
Let $\rho$ be an irreducible representation of $\GG_q$. By the orbit method,  $ \rho$ is obtained from a character $\theta\in\widehat{\g}_q$, whose restriction to $\ZZ(\g_q)$ coincides with the central character of $\rho$.
Assume that $\theta=\psi_\ba$ for some $\ba\in\F_q^{m+l_2+l_3}$, whose entries are indexed as in~\eqref{vec-assigne}.
Our next goal is to prove that 
\begin{equation}\label{dim-computation}
\dim\rho=q^{\frac{\rank_{\F_q}(F_\g(a_1,\dots, a_m))}{2}}=q^{\frac{\rank_{\F_q}(F_\g(\pr_3(\ba)))}{2}}.
\end{equation}
The proof is similar to the argument of \cite[Lemma 3.3]{Voll}, but for the reader's convenience we provide some details. 
Proposition \ref{stab} implies that $\ZZ(\g_q)\subseteq \mathrm{Stab}_{\GG_q}(\theta)$.
Since $p>C_3$, it follows that $K$ is invertible in $\F_q$, so that  after tensoring by $\F_q$ the $\vv_i$'s form a basis of $\g_q/\ZZ(\g_q)$.
For
$x=\sum_{i=1}^nx_i\vv_i\in \mathrm{Stab}_{\GG_q}(\theta)$ 
and $y=\sum_{i=1}^ny_i\vv_i\in\g_q/\ZZ(\g_q)$ we
have $\theta([x,y])=1$.
From~\eqref{commutator}  
and the fact that $p>C_1$ it follows that
\begin{equation*}
\begin{split}
\psi\left(\sum_{i=1}^n\left(\sum_{1\leq r<i}\sum_{k=1}^m a_k\lambda^k_{ri}x_r-\sum_{i<s\leq n}\sum_{k=1}^m a_k\lambda^k_{is}x_s\right)y_i\right)
&=\psi_{\bf a}\left(\sum_{1\leq i<j\leq n}\sum_{k=1}^m \lambda^k_{ij}(x_iy_j-x_jy_i)\ww_k\right)
=1.
\end{split}
\end{equation*}
Since  $\psi$ is a primitive character, 
it follows that $\mathrm{Stab}_{\GG_q}(\theta)/\ZZ(\g_q)$ is defined by the linear equations
\begin{equation*}
\sum_{i<s\leq n}\sum_{k=1}^m a_k\lambda^k_{is}x_s-\sum_{1\leq r<i}\sum_{k=1}^m a_k\lambda^k_{ri}x_r=0, \qquad 1\leq i\leq n.
\end{equation*} 
Consequently, $x=\sum_{i=1}^nx_i\vv_i\in\mathrm{Stab}_{\GG_q}(\theta)/\ZZ(\g_q)$ if and only if $(x_1,\dots,x_n)\in \ker F_\g(a_1,\dots, a_m)$. The last statement implies that 
$\#\mathrm{Stab}_{\GG_q}(\theta)=q^{\dim_{\F_q} (\g_q)-\rank_{_{\F_q}}(F_\g(a_1,\dots, a_m))}$. 
Equality~\eqref{dim-computation}
now follows from~\eqref{dim-Kirillov}.

\begin{definition}[(Admissible sets of vectors)] 
A set of vectors 
$$\left\{ \ba_\ell\in\F_q^{m+l_2+l_3}:\,\, 1\leq\ell\leq (l_1+l_2)f\right\},$$
is called an {\it admissible set of vectors} if $\{\pr_2(\ba_\ell): 1\leq \ell\leq (l_1+l_2)f\}$ is a basis of the $\F_p$-vector space $\F_q^{l_1+l_2}$.
\end{definition} 

Now let $\tilde{\rho}$ be a faithful representation of $\GG_q$ with the smallest possible  dimension. Note that the dimension of $\Omega_1(\ZZ(\g_q))=\ZZ(\g_q)$ over $\F_p$ is $(l_1+l_2)f$ and  $\ZZ(\GG_q)\cong \ZZ(\g_q)$. Therefore 
\begin{equation*}
\widehat{\Omega}_1(\ZZ(\GG_q))\cong\widehat{\Omega}_1(\ZZ(\g_q))\cong\bigoplus_{\ell=1}^{l_1+l_2}\F_q.
\end{equation*}
 Thus by Lemma~\ref{Meyer} the representation $\tilde{\rho}$ decomposes into $(l_1+l_2)f$ irreducible representations
$$\tilde{\rho}=
\bigoplus_{\ell=1}^{(l_1+l_2)f}\rho_{\ba_\ell},\qquad
\ba_\ell\in\F_q^{m+l_2+l_3},
$$
where the representation $\rho_{\ba_\ell}$ is obtained by $\psi_{\ba_\ell}\in\widehat{\g}_q$.  
Since the restriction of $\psi_{\ba_\ell}$ is the central character of $\rho_{\ba_\ell}$, it follows from Lemma~\ref{Meyer} that 
the set \[
\left\{\psi_{\ba_\ell}\big|_{\widehat{\Omega}_1(\ZZ(\g_q))}: 1\leq \ell\leq (l_1+l_2)f\right\}
\] is a basis of $\widehat{\Omega}_1(\ZZ(\g_q))$ and therefore the set
\[
\left\{ \pr_2(\ba_\ell): 1\leq \ell\leq (l_1+l_2)f \right\}
\] 
is a basis of the $\F_p$-vector space $\F_q^{l_1+l_2}$.
To summarize, we have proven that for each faithful representation $\tilde{\rho}$ with the smallest possible dimension, we can find an admissible set of vectors 
\[
\left\{ \ba_\ell\in\F_q^{m+l_2+l_3}:\,\, 1\leq\ell\leq (l_1+l_2)f\right\}
\]
such that  
\begin{equation}\label{dim-computation2}
\dim(\tilde{\rho})=\sum_{\ell=1}^{(l_1+l_2)f} q^{\frac{\rank_{\F_q}(F_\g(\pr_3(\ba))}{2}}.
\end{equation} 
Conversely, let 
$\left\{\ba_\ell\in\F_q^{m+l_2+l_3}: 1\leq \ell\leq (l_1+l_2)f\right\}
$
be an admissible set of vectors. Then by Lemma~\ref{central-span-faith}, we can construct a faithful representation $\tilde{\rho}$, not necessarily 
of minimal dimension,  such that its dimension is equal to~\eqref{dim-computation2}. In the definition of admissible vectors we considered $\F_q^{l_1+l_2}$ as an $\F_p$-vector space. We now consider $\F_q^{l_1+l_2}$ as an $\F_q$-vector space and define the following notion.

\begin{definition}[(Regular sets of vectors)] 
A set of vectors  \[
\left\{ \ba_\ell\in \F_q^{m+l_2+l_3}:\, 1\leq \ell\leq  l_1+l_2\right\},
\]
is called a  {\it regular set of vectors} if the set $\left\{\pr_2(\ba_\ell): 1\leq \ell\leq (l_1+l_2)\right\}$ is a basis of the $\F_q$-vector space $\F_q^{l_1+l_2}$. 
\end{definition}

We now claim that 
\begin{equation}\label{f.min}
\mf(\GG_q)=\min\left\{\sum_{\ell=1}^{l_1+l_2} fq^{\frac{\rank_{\F_q}(F_\g(\pr_3(\ba_\ell)))}{2}}: \left\{ \ba_\ell\in \F_q^{m+l_2+l_3}\right\}_{\ell=1}^{l_1+l_2} \, \text{ is  a  regular set}\right\}.
\end{equation}

Let $\{\omega_1,\dots,\omega_f\}$ be a basis of $\F_q$ over $\F_p$ and let 
\[
\left\{ \ba_\ell\in \F_q^{m+l_2+l_3}:\, 1\leq \ell\leq l_1+l_2\right\}
\] be a regular set of vectors that minimizes~\eqref{f.min}. Clearly $\{\omega_i\ba_\ell: 1\leq i\leq f,\, 1\leq \ell\leq l_1+l_2\}$ is an admissible set of vectors and 
$$
\sum_{\ell=1}^{l_1+l_2} fq^{\frac{\rank_{\F_q}(F_\g(\pr_3(\ba_\ell)))}{2}}=\sum_{i=1}^f\sum_{\ell=1}^{l_1+l_2} q^{\frac{\rank_{\F_q}(F_\g(\pr_3(\omega_i\ba_\ell)))}{2}}\geq \mf(\GG_q).
$$
Conversely, let $\tilde{\rho}$ be a faithful representation with the smallest possible dimension.  From the above discussion 
we obtain 
an admissible set of vectors 
\[
\left\{ \ba_\ell\in\F_q^{m+l_2+l_3}:\,\, 1\leq\ell\leq (l_1+l_2)f\right\}.
\] From Lemma~\ref{blocks} this set can be partitioned into $f$ sets $\mathcal{B}_i$ in which each $\mathcal{B}_i$ is a regular set of vectors. Without loss of generality assume that 
$$\sum_{\ba_\ell\in\mathcal{B}_1} q^{\frac{\rank_{\F_q}(F_\g(\pr_3(\ba_\ell)))}{2}}\leq \sum_{\ba_\ell\in\mathcal{B}_i} q^{\frac{\rank_{\F_q}(F_\g(\pr_3(\ba_\ell)))}{2}} \qquad 2\leq i\leq f.$$
Thus
\begin{equation*}
\sum_{\ba_\ell\in\mathcal{B}_1} fq^{\frac{\rank_{\F_q}(F_\g(\pr_3(\ba_\ell)))}{2}}\leq \dim(\tilde{\rho}),
\end{equation*}
which proves the claim. 

We are now ready to finish the proof of Theorem~\ref{Algorithm}. Let $\{ \ba_\ell\in \F_q^{m+l_2+l_3}:\, 1\leq \ell\leq l_1+l_2\}$ be a regular set of vectors. 
Let $A\in\mathrm{GL}_{l_1+l_2}(\F_q)$ be the matrix whose rows are 
$$\pr_2(\ba_1),\ldots,\pr_2(\ba_{l_1+l_2}).$$
 Therefore 
\begin{equation*}
A=
\begin{footnotesize}
\begin{pmatrix}
a_{1 1} & \dots & a_{1 l_1} & b_{1 1}& \dots & b_{1 l_2}\\
\vdots & \ddots & \vdots & \vdots & \ddots & \vdots\\
a_{l_1 1} & \dots & a_{l_1 l_1} & b_{l_1 1}& \dots & b_{l_1 l_2}\\[.3cm]
a_{(l_1+1) 1} & \dots & a_{(l_1+1) l_1} & b_{(l_1+1) 1}& \dots & b_{(l_1+1) l_2}\\
\vdots & \ddots & \vdots & \vdots & \ddots & \vdots\\
a_{(l_1+l_2) 1}& \dots & a_{(l_1+l_2) l_1} & b_{(l_1+l_2) 1}& \dots & b_{(l_1+l_2) l_2}
\end{pmatrix}
\end{footnotesize}
,
\end{equation*} 
where $\pr_2(\ba_i)=(a_{i1},\ldots,a_{il_1},b_{i1},\ldots,b_{il_2})$.
The first $l_1$ columns of $A$ are linearly independent  over $\F_q$,  and therefore we can find an invertible $l_1\times l_1$ submatrix of the first $l_1$ columns.
By possibly permuting the rows of $A$ we can assume that this submatrix lies at the intersection of the first $l_1$ rows and $l_1$ columns of $A$. 
It is clear that 
\begin{equation*}
\sum_{\ell=1}^{l_1+l_2} q^{\frac{\rank_{\F_q}(F_\g(a_{\ell 1},\dots, a_{\ell m}))}{2}} \ge \sum_{\ell=1}^{l_1} q^{\frac{\rank_{\F_q}(F_\g(a_{\ell 1},\dots, a_{\ell m}))}{2}}+ l_2.
\end{equation*}
From this and~\eqref{f.min} we conclude that 
\begin{equation}\label{ineq-con-1}
\mf(\GG_q)\geq \min\left\{ \sum_{\ell=1}^{l_1} fq^{\frac{\rank_{\F_q}(F_\g(a_{\ell 1},\dots,a_{\ell m}))}{2}}: \begin{pmatrix}
a_{1 1} & \dots & a_{1 l_1}\\
\vdots  & \ddots & \vdots\\
 a_{l_1 1} & \dots & a_{l_1 l_1}
\end{pmatrix}\in\GL_{l_1}(\F_q)\right\}+fl_2.
\end{equation}
Conversely, let 
\begin{equation*}
\{a_{\ell (l_1+1)},\dots,a_{\ell m}\}_{1\leq \ell\leq l_1}
\end{equation*}
be an arbitrary set of elements of $\F_q$ and let
$$
B:=\begin{pmatrix}
a_{1 1} & \dots & a_{1 l_1}\\
\vdots  & \ddots & \vdots\\
 a_{l_1 1} & \dots & a_{l_1 l_1}
\end{pmatrix}
\in\GL_{l_1}(\F_q)
$$
be an arbitrary invertible matrix. Then the  rows of the matrix 
\begin{equation*}
\begin{pmatrix}
B & 0\\
0&I_{l_1\times l_1}
\end{pmatrix}
\in \GL_{l_1+l_2}(\F_q)
\end{equation*}
are projections (under $\pr_2$) of 
 a regular set of vectors in $\F_q^{m+l_2+l_3}$. Similar to the proof of~\eqref{f.min}, 
using this regular set of vectors we can construct a faithful representation of $\GG_q$ of dimension 
$$
\sum_{\ell=1}^{l_1} fq^{\frac{\rank_{\F_q}(F_\g(a_{\ell 1},\dots, a_{\ell m}))}{2}}+ fl_2.
$$
From this we conclude that 
\begin{equation}\label{ineq-con-2}
\mf(\GG_q)\leq \min\left\{ \sum_{\ell=1}^{l_1} fq^{\frac{\rank_{\F_q}(F_\g(a_{\ell 1},\dots,a_{\ell m}))}{2}}: \begin{pmatrix}
a_{1 1} & \dots & a_{1 l_1}\\
\vdots  & \ddots & \vdots\\
 a_{l_1 1} & \dots & a_{l_1 l_1}
\end{pmatrix}\in\GL_{l_1}(\F_q)\right\}+fl_2.
\end{equation}
Therefore, by combining~\eqref{ineq-con-1} and~\eqref{ineq-con-2} we obtain Theorem~\ref{Algorithm}.
\bigskip

\noindent We now address Example~\ref{cubic-curve-example}, Example~\ref{Binary quadratic form},  Example~\ref{Binary cubic form}, and
Example~\ref{ex:verticalprime} in detail. 
By a straightforward calculation, one can verify that in all of these examples $C_1=C_2=C_3=1$, hence Theorem
\ref{Algorithm} is applicable for $p>2$.

\subsection{Details for Example~\ref{cubic-curve-example}}
\label{cubic-curve-example-sol}
From the defining bracket relations we deduce that $\g_a'=\ZZ(\g_a)=\Span_\Z\{v_7,v_8,v_9\}$ and so $\g_a$ is a 2-step nilpotent Lie algebra. From the relations we also obtain the following commutator matrix:
\begin{equation*}
F_\g(T_1,T_2,T_3)=\begin{pmatrix}
0 & M\\
-M^{\mathrm{tr}} & 0
\end{pmatrix}, \qquad M:=M(T_1,T_2,T_3)=\begin{pmatrix}
T_1 & T_2 & aT_3\\
T_3 & T_1 & T_2\\
T_3 & 0 & T_1
\end{pmatrix}.
\end{equation*}
Observe that the determinant of $M$ is 
\begin{equation*}
g(T_1,T_2,T_3):=T_3T_2^2+T_1^3-T_1T_2T_3-aT_1T_3^2.
\end{equation*}
By Theorem~\ref{Algorithm} the faithful dimension of $\GG_{a,p}=\exp(\g_a\otimes_\Z\F_p)$, for $p\geq 3$, is the minimum value of
\begin{equation}\label{BI}
p^{\rank_{\F_p}(M(x_{11},x_{12},x_{13}))}+p^{\rank_{\F_p}(M(x_{21},x_{22},x_{23}))}+p^{\rank_{\F_p}(M(x_{31},x_{32},x_{33}))}
\end{equation}
subject to the  condition
\begin{equation}
\label{condDeT}\begin{pmatrix}
x_{11} & x_{12} & x_{13}\\
x_{21}  & x_{22} & x_{23}\\
x_{31}  & x_{32}  & x_{33}
\end{pmatrix}\in \GL_3(\F_p).
\end{equation}
Let $p$ be a prime not dividing $a$. Computing all $2\times 2$ minors of $M$ shows that 
\[
2\leq \rank_{\F_p}(M(x,y,z))\leq 3
\] unless $x=y=z=0$. Let us consider the question of existence of a vector $(x,y,z)\in \F_p^3$ such that $x\neq 0$ and 
\begin{equation}\label{a}
g(x,y,z)=y^2z+x^3-xyz-axz^2=0.
\end{equation}
Obviously we should take $z\neq 0$
Picking $x=1$, we obtain the equation 
$zy^2-zy+(1-az^2)=0$,
whose discriminant with respect to $y$ is equal to  $4az^3+z^2-4z$. Thus to solve \eqref{a} it suffices to show that for any non-zero $a\in \Z$ the curve 
$
Y^2=4aX^3+X^2-4X
$
has a rational point in $\F_p$ with $X\neq 0$. This can be done by noticing that $Y^2-4aX^3-X^2+4X$ is absolutely irreducible~\cite[Corollary, p.13]{Schmidt} and thus by Hasse's bound on the number of $\F_p$-points of elliptic curves
~\cite[Theorem 2A, p.10]{Schmidt} one can verify that  such a point exists for $p>1800$.
Let $(1,y,z)$ be a solution of~\eqref{a}. 
Then the vectors $(0,1,0)$ and $(0,0,1)$ and $(1,y,z)$ satisfy~\eqref{condDeT}, and minimize~\eqref{BI}. Thus the faithful dimension of $\GG_{a,p}$ is equal to $3p^2$.

\subsection{Details for Example~\ref{Binary quadratic form} and Example~\ref{ex:verticalprime}} 
\label{Binary quadratic form-sol}

We discuss these examples together by proving the following formula
\begin{equation*}
\mf(\GG_q)=
\begin{cases}
2fq &\text{if }\, p\equiv 1\pmod{4};\\
2fq &\text{if }\, p\equiv 3\pmod{4} \, \text{ and $f$ is even};\\
2fq^2 &\text{if }\, p\equiv 3\pmod{4} \, \text{ and $f$ is odd}.
\end{cases}
\end{equation*}

The commutator relations imply that $\g'=\ZZ(\g)=\Span_\Z\{v_5,v_6\}$, and that $\g$ is a 2-step nilpotent Lie algebra.  The commutator matrix of $\g$ can be easily seen to be
\begin{equation*}
F_\g(T_1,T_2)=
\begin{pmatrix}
0 & T_1 & 0 & T_2\\
-T_1 & 0 & T_2 & 0\\
0 & -T_2 & 0 & T_1\\
-T_2 & 0 & -T_1 & 0
\end{pmatrix}.
\end{equation*}
Note that $\det F_\g(T_1,T_2)=(T_1^2+T_2^2)^2$.  
By Theorem~\ref{Algorithm} the faithful dimension of $\GG_q$ is given by
$$
\min\left\{ fq^{\frac{\rank_{\F_q}(F_\g(x_{11},x_{12}))}{2}}+fq^{\frac{\rank_{\F_q}(F_\g(x_{21},x_{22}))}{2}}: \begin{pmatrix}
x_{11} & x_{12}\\
x_{21}  & x_{22}
\end{pmatrix}\in\GL_2(\F_q)\right\}.
$$ 
For $p\equiv 1 \pmod{4}$, let $\alpha$ denote a square root of $-1$ in $\F_q$.  Then the trivial lower bound $2fq$ can be realized by the choice of vectors $(\alpha,1)$ and $(-\alpha,1)$. 
We now consider the  case $p\equiv 3 \pmod{4}$. For these primes, observe that $-1$ is a square in $\F_q$ if and only if $f$ is even. By  the above argument the faithful dimension of $\GG_q$ is $2fq$ when $f$ is even. 
Now suppose that $f$ is odd. 
Then $-1$ is not a square in $\F_q$, and therefore $\det F_\g(a_1,a_2) \neq 0$ for all non-zero vectors $(a_1, a_2) \in \F_q^2$.
This implies that 
$$\rank_{\F_q}(F_\g(a_1,a_2))=4\,\text{ for all } 0\neq (a_1,a_2)\in \F_q^2.$$
Therefore the faithful dimension of $\GG_q$ is at least $2fq^2$, which can be realized by the standard basis. 

\subsection{Details for Example~\ref{Binary cubic form}} 
\label{Binary cubic form-sol}
The commutator relations imply that $\g'=\ZZ(\g)=\Span_\Z\{v_7,v_8\}$, implying that $\g$ is a 2-step nilpotent Lie algebra with the commutator matrix given by 
\begin{equation}
\label{eq:MMtt}
F_\g(T_1,T_2)=\begin{pmatrix}
0 & M\\
-M^{\mathrm{tr}} & 0
\end{pmatrix},\ \text{ where } M:=M(T_1,T_2)=
\begin{pmatrix}
T_1 & T_2 & 0\\
0 & T_1 & T_2\\
-T_2 & T_2 & T_1
\end{pmatrix}.
\end{equation}
By Theorem~\ref{Algorithm} the faithful dimension of $\GG_p$ is given by 
\begin{equation*}
\min \left\{  p^{\rank_{\F_p}(M(x_{11},x_{12}))}+p^{\rank_{\F_p}(M(x_{21},x_{22}))}
: \begin{pmatrix}
x_{11} & x_{12}\\
x_{21}  & x_{22}
\end{pmatrix}\in\GL_2(\F_p) \right\}. 
\end{equation*}

A simple inspection of $2\times 2$ minors of $M$ shows that for a non-zero vector $(T_1, T_2)$, the matrix $M(T_1,T_2)$ 
has rank at least $2$, implying $\mf(\GG_p) \ge 2p^2$. Note that $\det M= T_1^3-T_1T_2^2-T_2^3$ is the homogenization of the polynomial $T_1^3-T_1-1$. This leads us to consider the number of roots of $f(T)=T^3-T-1$ over $\F_p$.

At this point we will make a digression and consider the more general question of determining the number of roots of a given integer polynomial 
over finite fields. Let $C_f$ be the companion matrix of a given polynomial $f(T) \in \ZZ[T]$, and set $M(T_1,T_2)=T_1 I_{d\times d} -T_2 C_f$, where  $d:=\deg f(T)$. The determinant of $M(T_1,T_2)$  is the homogenization of $f(T)$. This construction leads to a general collection of interesting examples of 2-step nilpotent Lie algebras with commutator matrix as in
\eqref{eq:MMtt}.
We refer the reader to Serre's book~\cite[Section 2.1.2]{SerreNX} or his beautiful paper~\cite{Serre-Jordan} for more details on what follows. 
  
Let $f(T)\in\Z[T]$ be a monic integer polynomial. The discriminant of $f(T)$ is defined to be
$$\mathrm{Disc}_f=\Delta_f^2, \qquad \Delta_f=\prod_{1\leq i<j\leq n}(\alpha_i-\alpha_j),$$
where $\alpha_1,\dots,\alpha_n$ are the roots of $f(T)$ in an algebraic closure of $\Q$. Note that since $f(T)$ is a monic polynomial, the discriminant $\mathrm{Disc}_f$ is in $\Z$. Henceforth, $p$ will denote an odd prime which does not divide $\mathrm{Disc}_f$.  
Denote the reduction of $f(T)$ modulo $p$ by $\bar{f}$. Then the roots of $\bar{f}(T)$ are also simple. Define $\mathcal{O}_f= \mathbb{Z}[\alpha_1 ,\dots,\alpha_n]$ and let $\mathfrak{p}$ be a prime ideal of $\mathcal{O}_f$ such
that $\mathfrak{p}\cap\mathbb{Z} = p\mathbb{Z}$. Such an ideal exists since $\mathcal{O}_f$ is integral over $\mathbb{Z}$. For such a prime we can define a unique element in the Galois group of $f$, which is called the Frobenius automorphism. 

\begin{theorem}[(Dedekind)]\label{Dedekind}
Let $E=\Q(\alpha_1,\dots,\alpha_n)$ be the splitting field of $f(T)$. There exists a unique element $\sigma_\mathfrak{p}\in {\rm Gal}(E/\mathbb{Q})$ such that
$\sigma_\mathfrak{p} (\alpha)\equiv \alpha^p \mod\mathfrak{p}$, 
for all $\alpha\in \mathcal{O}_f$. Moreover if $\bar{f}(T)=f_1(T)\cdots f_g(T)$ with $f_i$ irreducible over $\mathbb{F}_p$ of degree $n_i$, then $\sigma_\mathfrak{p}$,
when viewed as a permutation of the roots of $f$, has the cyclic decomposition $\sigma_1\cdots \sigma_g$ with $\sigma_i$ a cycle of length $n_i$.
\end{theorem}
\begin{proof}
See~\cite[Theorems 4.37 and 4.38]{Jacobson}
\end{proof}
For a monic integer polynomial $f(T)\in\Z[T]$ define 
$$N_f(p):= \# \{ a\in\mathbb{F}_p:\, f(a) =0 \}.$$
Theorem~\ref{Dedekind} shows that $N_f(p)$ also counts the number of fixed points of $\sigma_\mathfrak{p}$ permuting the roots of $f$. 

Let $\mathcal{O}_E$ be the ring of integers of $E$ and $\mathfrak{P}$ be
a prime ideal of $\mathcal{O}_E$ such that $\mathfrak{P}\cap \Z =p\Z$, where $p$ does not divide the discriminant of $E$.  Then, as above, one can prove the existence of a unique automorphism $\sigma_{\mathfrak{P}}\in\mathrm{Gal}(E/\Q)$ such that $\sigma_{\mathfrak{P}}(\alpha)\equiv \alpha^p \pmod{\mathfrak{P}}$ for all $\alpha\in\mathcal{O}_E$. Let $\mathfrak{P}\cap \mathcal{O}_f=\mathfrak{p}$. Since the elements of $\mathrm{Gal}(E/\Q)$ are uniquely determined by
their restrictions to $\mathcal{O}_f$, we have $\sigma_{\mathfrak{P}}=\sigma_{\mathfrak{p}}$. The automorphism $\sigma_{\mathfrak{p}}$ is called the {\it Frobenius automorphism} and it describes the splitting behaviour of the prime $p$. 
It is well known that  $p$ splits completely in $E$ if and only if $\sigma_\mathfrak{p}$ is the identity element.  Now let $\mathfrak{P}$ and $\mathfrak{P}'$ be two primes in $\mathcal{O}_E$ lying above the rational prime $p$.  
One can show 
(see~\cite[Section 9]{Neukirch} for more details)
that there exists $\tau\in\mathrm{Gal}(E/\Q)$ such that $\tau\sigma_{\mathfrak{P}}\tau^{-1}=\sigma_{\mathfrak{P}'}$. This implies that the conjugacy class of $\sigma_\mathfrak{P}$ is independent of the choice of $\mathfrak{P}$.  Let us now turn to the question of computing $N_f(p)$ when $f(T)$ is a cubic polynomial. The following proposition relates the Legendre symbol of the discriminant of $f$ to the number of the irreducible factors of $\bar{f}$.  
\begin{proposition}\label{dis/p}
 Let $f(T)\in \mathbb{Z}[T]$ be a monic irreducible polynomial of degree $n$ with the discriminant $D$, and suppose $p$ is an odd prime which does not divide the discriminant of $f$. If $\bar{f}=f_1\cdots f_g$ with $f_i$ irreducible over $\mathbb{F}_p$ then $(\frac{D}{p})=(-1)^{n-g}$, 
where $(\frac{\cdot}{p})$ is the Legendre symbol.
\end{proposition}
\begin{proof}
Continuing to use the same notation as before, we denote the splitting field of $f$ by $E$ and its ring of integers by $\mathcal{O}_E$. Set $D=\mathrm{Disc}_f$ and set $K=\mathbb{Q}(\sqrt{D})$ which is a subfield of $E$. Let $\mathfrak{p}$ be a prime in $\mathcal{O}_E$ lying over $p$ and write  $\wp=\mathfrak{p}\cap K$. Then $\sigma_\wp:=\sigma_{\mathfrak{p}|_K}$ is the Frobenius automorphism assigned to $\wp$ in $K/\mathbb{Q}$. Suppose $\sigma_\mathfrak{p}$ is an even permutation. Then $\sigma_\mathfrak{p}(\Delta_f)=\Delta_f$ and so $\sigma_\wp$ is trivial over $K$, which implies that $(\frac{D}{p})=1$. If, on the other hand, $\sigma_\mathfrak{p}$ is an odd permutation, then $\sigma_\mathfrak{p}(\Delta_f)=-\Delta_f$, and $\sigma_\wp$ is not-trivial, implying that $(\frac{D}{p})=-1$. We have thus shown that  ${\rm sgn}(\sigma_\mathfrak{p})=(\frac{D}{p})$. 
Let $n_i= \deg f_i$. Viewing $\sigma_\mathfrak{p}$ as a permutation of the roots of $f$, from Theorem~\ref{Dedekind},  we obtain
$${\rm sgn}(\sigma_\mathfrak{p})=(-1)^{\sum_{i=1}^{g} (n_i-1)}=(-1)^{n-g},$$
since $\sum_{i=1}^g n_i=n$. This finishes the proof.
\end{proof}
 As an application we obtain
\begin{corollary}
Let $f(T)\in \mathbb{Z}[T]$ be an irreducible monic cubic polynomial with discriminant $D$, and suppose $p$ is an odd prime which does not divide $D$. Then
$$N_f(p)=
\begin{cases}
0 {\,\, {\rm or}\,\,} 3  & \mbox{if } (\frac{D}{p})=1; \\
1 & \mbox{if } (\frac{D}{p})=-1.
\end{cases}
$$
\end{corollary}
We now turn to the special case $f(T)=T^3-T-1$. The discriminant of $f$ is $-23$ and then by the quadratic reciprocity we deduce that for $p\neq 23$
$$N_f(p)=
\begin{cases}
0 {\,\, {\rm or}\,\,} 3  & \mbox{if } (\frac{p}{23})=1; \\
1 & \mbox{if } (\frac{p}{23})=-1.
\end{cases}
$$
When $(\frac{p}{23})=1$, we will need the reduction theory of binary quadratic forms to determine $N_f(p)$. We refer the reader to~\cite[Chapter 2, $\S 8$]{Flath} for more details. Let $\Delta<0$ be an integer and assume that $\Delta\equiv 0, 1 \pmod{4}$. The modular group $\mathrm{SL}_2(\Z)$ acts on 
$$
\Sigma_\Delta:=\left\{g(x,y)=ax^2+bxy+cy^2: a,b,c\in \Z, \, a>0,\, \gcd(a,b,c)=1,\, b^2-4ac=\Delta\right\},
$$  
by linear change of variables.
By the reduction theory of positive definite integral binary quadratic forms (for example see \cite[Theorem 3.9]{Cox}), the number $h(\Delta)$ of $\mathrm{SL}_2(\Z)$-orbits is finite. This number is called the {\it class number} of $\Delta$. In the case at hand, we have $h(-23)=3$, and $\mathrm{SL}_2(\Z)$-orbits of $\Sigma_{-23}$ are represented by the forms $x^2+xy+6y^2$ and $2x^2\pm xy+3y^2$. Note that $2x^2+xy+3y^2$ and $2x^2-xy+3y^2$ are $\mathrm{GL}_2(\Z)$-equivalent and thus represent the same set of integers. It is easy to show that 
$(\frac{p}{23})=1$,
 if and only if $p$ is represented by exactly one of the form $x^2+xy+6y^2$ or $2x^2+xy+3y^2$ (see~\cite[Proposition 10.2]{Flath}).    
Let $L$ be the cubic extension of $\Q$ obtained by adding a root of $f(T)$ and set $K=\mathbb{Q}(\sqrt{-23})$:
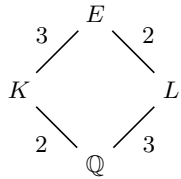
\begin{center}
\begin{footnotesize}
\begin{tikzpicture}[auto,node distance=1.5cm]
\node (B) at (0,1) {$E$};
\node (C) at (-1,0) {$K$};
\node (D) at (1,0) {$L$};
\node (E) at (0,-1) {$\mathbb{Q}$};
\draw[semithick] (B) --node[above left]{$3$} (C);
\draw [semithick] (B) --node[above right]{$2$} (D);
\draw [semithick] (C) --node[below left]{$2$} (E);
\draw [semithick] (D) --node[below right]{$3$} (E);
\end{tikzpicture}
\end{footnotesize}
\end{center}
Note that $\mathrm{Gal}(E/\mathbb{Q})=S_3$. For $p\neq 23$, set $p\mathcal{O}_K=\mathfrak{p}\overline{\mathfrak{p}}$. Since the class number of $K$ is $3$ and $E$ is unramified over $K$,  $E$ is the Hilbert class field of $K$, i.e. the maximal unramified abelian extension of $K$. From this we can conclude that $\mathfrak{p}$ splits completely in $E$ if and only if $\mathfrak{p}$ is a principal ideal~\cite[Corollary 5.25]{Cox}. Moreover, note that the ring of integers of the quadratic extension $K$ is $\Z[(1+\sqrt{-23})/2]$ and so $\mathfrak{p}$ is a principal ideal if and only if $p=x^2+xy+6y^2$.
Putting all these together, we conclude that $p=x^2+xy+6y^2$ if and only if $p$ splits completely in $E$. This means that $T^3-T-1$ has $3$ roots in $\F_p$ if and only if $p=x^2+xy+6y^2$. This also shows that $p=2x^2+xy+3y^2$ if and only if $T^3-T-1$ has no root in $\F_p$. 
Consequently,
\begin{equation}\label{Np}
N_f(p)=
\begin{cases}
1 &\text{if }\, \left(\frac{p}{23}\right)=-1;\\
0 &\text{if }\, p\,\, \text{is of the form\,}  2x^2+xy+3y^2; \\
3 &\text{if }\, p\,\, \text{is of the form\,} x^2+xy+6y^2.
\end{cases}
\end{equation}  
Let $N_{\bf X}(p)$ denote the number of rational points of the projective variety  ${\bf X}:=T_1^3-T_1T_2^2-T_2^3=0$ in $\mathbf{P}^1(\F_p)$. Then from~\eqref{Np}, for all $p\neq 23$ we obtain 
\begin{equation*}
N_{\bf X}(p)=
\begin{cases}
1 &\text{if }\, \left(\frac{p}{23}\right)=-1;\\
0 &\text{if }\, p\,\, \text{is of the form\,}  2x^2+xy+3y^2; \\
3 &\text{if }\, p\,\, \text{is of the form\,} x^2+xy+6y^2.
\end{cases}
\end{equation*} 
When $(\frac{p}{23})=-1$, we have $N_{\bf X}(p)=1$ and hence the faithful dimension of $\GG_p$ equals $p^2+p^3$. When $p$ is of the form  $2x^2+xy+3y^2$ then $N_{\bf X}(p)=0$ and so the minimum is exactly $2p^3$. This implies that the faithful dimension is $2p^3$. In the remaining case, we can find distinct points $(x_{11},x_{12})$ and $(x_{21},x_{22})$ in ${\bf X}$. Since $M(x_{11},x_{12})$ and $ M(x_{21},x_{22})$ have both rank $2$, 
it follows that in this case the faithful dimension is $2p^2$. Moreover $T_1^3-T_1-1$ has a double root and a simple root in $\F_{23}$. Thus the same argument shows that in this case the faithful dimension is $2(23)^2$.

\subsection{Details for Example~\ref{LeeLie}}
\label{LeeLie-sol}
 The commutator relations imply that $\g'=\ZZ(\g)=\Span_\Z
\{v_6,v_7,v_8\}$, implying that $\g$ is a 2-step nilpotent Lie algebra with the commutator matrix given by 
\begin{equation*}
F_\g(T_1,T_2,T_3)=
\begin{pmatrix}
0 & M\\
-M^{\mathrm{tr}} & 0
\end{pmatrix}, \qquad M:=M(T_1,T_2,T_3)=
\begin{pmatrix}
T_1 & T_2\\
T_3 & T_1\\
2T_2 & T_3
\end{pmatrix}.
\end{equation*}
For a given odd prime $p$, and a nonzero vector $(T_1,T_2,T_3)\in\F_p^3$, the rank over $\F_p$ of $M$ is equal to $1$ if and only if 
$(T_1,T_2,T_3)$ is proportional to 
$(\lambda^2/2,\lambda/2,1)$
such that $\lambda^3-2=0$, and is equal to $2$ otherwise. 
Set $f(\lambda)=\lambda^3-2$.
As noted in 
\cite[Corollary 2.3]{SLee}, 
\begin{equation*}
N_f(p)=
\begin{cases}
1 &\text{if }\, p\equiv 2 \pmod{3}\text{ or }p=3;\\
3 &\text{if }\, p\equiv 1 \pmod{3}\text{ and }p\text{ is represented by the form }x^2+27y^2;\\
0 &\text{if }\, p\equiv 1 \pmod{3}\text{ and }p\text{ is not represented by the form }x^2+27y^2.
\end{cases}
\end{equation*}  
The rest of the argument is similar to Example~\ref{Binary cubic form}.

\section{Proofs of Theorem \ref{rationality-Ax} and Theorem \ref{APsets}}\label{Ax-Section}
Before we begin the proofs of  Theorem~\ref{rationality-Ax}
and
Theorem~\ref{APsets},
let us recall the setting. As before, let $\g$ be a nilpotent $\Z$-Lie algebra of nilpotency class $c$
which is finitely generated as an abelian group, and set $q:=p^f$ for some $f\geq 1$.
 Let $F_\g(T_1,\dots,T_m)$ denote the matrix of linear forms that is defined in~\eqref{commutating-matrix}. Since $F_\g(T_1,\dots,T_m)$ is a skew-symmetric $n\times n$ matrix it follows that for all $x_1,\dots,x_m\in\F_q$, the rank of $F_\g(x_1,\dots,x_m)$ is an even number no larger than $n$. Let $M$ be the set of all integer vectors $\mu=(a_1,\dots,a_{l_1})\in\Z^{l_1}$, with $0\leq a_i\leq n/2$, and assign to each $\mu \in M$ the polynomial of degree at most $n/2$ given by
$$
g_\mu(T)=T^{a_1}+\dots+T^{a_{l_1}}+l_2.
$$
Since $g_{\mu}$ is symmetric in $a_1, \dots, a_{l_1}$, we will only consider those integer vectors $\mu$ with \[a_1\leq \dots\leq a_{l_1},\] and order them with the {\it reverse lexicographical order}, i.e. $\mu\lhd \mu'$ if the rightmost non-zero component of the vector $\mu'-\mu$ is positive.  If $\mu\lhd\mu'$ and $q>l_1$, then we can easily see that
\begin{equation}\label{ine-g}
g_\mu(q)<g_{\mu'}(q).
\end{equation}
Since $r= \# M< \infty$, we can sort its elements as $\mu_1\lhd\dots\lhd\mu_r$.
For a given vector $\mu$, define the following affine variety associated to $\mu=(a_1,\dots,a_{l_1})$:
\begin{equation*}
{\bf X}_\mu:=\left\{(x_{ij})\in \mathrm{M}_{l_1,m}(\C):  \rank_\C(F_\g(x_{i1},\dots, x_{im})) =2 a_i,\quad \det\begin{pmatrix}
x_{1 1} & \dots & x_{1 l_1}\\
\vdots  & \ddots & \vdots\\
x_{l_1 1} & \dots & x_{l_1 l_1} 
\end{pmatrix}\neq 0\right\}.
\end{equation*}
Note that the non-vanishing condition on the determinant can be turned into an equation by introducing a new variable, standing for the inverse of the determinant. We also remark that ${\bf X}_\mu$ is defined over $\Z$ because $F_\g(T_1,\dots,T_m)$ is an integer matrix.

\subsection{Proof of Theorem \ref{rationality-Ax}}
\label{subsec:pfrationality-Ax}

For $\mu\in M$ set
$$
\Sigma_\mu:=\left\{p>\max\{l_1,c,C_1,C_2,C_3\}: {\bf X}_\mu(\F_p)\neq \emptyset\right\},
$$
where the $C_i$'s are as in 
Theorem \ref{Algorithm}.
For every  integer $k$ such that  $1\leq k\leq r$, Theorem~\ref{Algorithm} and~\eqref{ine-g} imply that
$\mf(\GG_p)=g_{\mu_k}(p)$
whenever 
$$p\in\mathscr{P}_k:=\Sigma_{\mu_k}\setminus \bigcup_{1\leq i<k}\Sigma_{\mu_i}.$$
 Since finite sets are Frobenius, the assertion of 
Theorem \ref{rationality-Ax} now follows from the following theorem due to Ax~\cite[Theorem 1]{Ax}.
\begin{theorem}[(Ax)]
With the above notation, $\Sigma_\mu$ is a Frobenius set. 
\end{theorem}

\begin{remark}
An analogue of Theorem \ref{rationality-Ax} holds for $\mf(\GG_{p^f})$ when
$f$ is fixed and $p$ varies. This statement can be established by a modification of the proof given above and applying \cite[Sec. 7.2.4, Ex.~2]{SerreNX}.
\end{remark}

\subsection{Proof of Theorem \ref{APsets}}
The proof of this theorem is similar to that of Theorem \ref{rationality-Ax}. Hence we will maintain the notation for the matrix $F_\g(T_1,\dots,T_m)$, the ordered set $M$, the polynomial $ g_\mu$, and the variety ${\bf X}_\mu$
as above. 
Now assume that $p>C$, where 
\begin{equation}
\label{eq:bigconstantnt}
C:=\max\{l_1,c,C_1,C_2,C_3\}.
\end{equation}
Consider the sets
$$
\Sigma'_\mu:=\left\{f \ge 1: {\bf X}_\mu(\F_{p^f})\neq \emptyset\right\}.
$$
It follows from a theorem of Dwork \cite[p.~6 and Section 4.3]{SerreNX} that there exists a function $\nu: \C \to \Z$ with finite support such that 
\begin{equation}\label{cardinality-dwork}
N_{ {\bf X}_{\mu}}(p^f):= \#  {\bf X}_\mu(\F_{p^f})= \sum_{z \in \C} \nu(z) z^f.  
\end{equation}
It is easy to see that the sequence $c_f=N_{ {\bf X}_{\mu}}(p^f)$  satisfies a linear recurrence relation of the form $ c_n= \sum_{k=1}^{r} a_k c_{n-k}$. We will now invoke the following theorem of Skolem-Mahler-Lech:

\begin{theorem}[(Skolem-Mahler-Lech), \cite{Poorten}]\label{SML}
Let $\{ u_n \}_{ n \ge 1}$ be a sequence of complex numbers satisfying a linear recurrence equation. Then its zero set
$\{ n: u_n = 0 \}$ is a union of a finite set and a finite number of 
sets of the form $n\equiv a\pmod b$ for integers $a,b$. 
\end{theorem}
One can easily verify that the sets of the form $F \cup A$, where $F$ is finite and $A$ is a finite union of arithmetic progressions form a Boolean algebra. 
To complete the proof of 
Theorem~\ref{APsets}, note that
from \eqref{ine-g} and 
Theorem~\ref{Algorithm} it follows that if $p>C$ and $f\in \mathscr{A}_k:=\Sigma'_{\mu_k}\setminus \bigcup_{1\leq i<k}\Sigma'_{\mu_i}$  then  
$\mf(\exp(\GG_q))=f  g_{k}(q)$.


\section{Free nilpotent Lie algebras}

In this section we will consider faithful representations of groups related to free nilpotent and free metabelian  Lie algebras. Let us recall some definitions. The free nilpotent Lie algebra of class $c$ on $n$ generators, denoted by $\f_{n,c}$, is the free object in the category of $n$-generated nilpotent Lie algebras (over $\Z$) of class $c$. More concretely, $\f_{n,c}$ can be constructed from the free Lie algebra on $n$ generators 
after quotienting out
the ideal generated by commutators of length $c+1$.

Recall that 
a Lie algebra $\mathfrak l$ is called metabelian if $[[\mathfrak l,\mathfrak l] ,[\mathfrak l,\mathfrak l]]=0$. 
 Similarly, one can define the {\it free metabelian Lie algebra} of class $c$ on $n$ generators as the free object in the category of $n$-generated metabelian Lie algebras of class $c$.

 For computational purposes, it will be convenient to work with {\it Hall bases} of free nilpotent Lie algebras. We will briefly review their constructions, and refer the reader to~\cite[Chapter II]{Bourbaki} or~\cite[Chapter IV]{Serre-Lie-Algebra}
for more details.

\subsection{Hall bases of free nilpotent Lie algebras} Our exposition of the notion of a Hall basis follows~~\cite[Chapter II]{Bourbaki} or~\cite[Chapter IV]{Serre-Lie-Algebra}. We will first need some basic definitions. 
A set $M$ with a map $M\times M\to M$ sending $(x,y)\mapsto [x,y]$ is called a {\it magma}. Let $X$ be a set and define inductively a family of sets $X_n$ $(n\geq 1)$ as follows: $X_1=X$ and $X_n=\amalg_{p+q=n}(X_p\times X_q)$. Let $M(X)$ denote the disjoint union $\amalg_{n=1}^\infty X_n$ and define $M(X)\times M(X) \to M(X)$ via 
$
X_p\times X_q\to X_{p+q}\subseteq M(X).
$
The magma $M(X)$ is called the {\it free magma} on $X$.  An element $w$ of $M(X)$ is
called a non-associative word on $X$. Its length, $\ell(w)$, is the unique $n$ such that $w\in X_n$.  
\begin{definition} A {\it Hall set} relative to $X$ is a totally ordered subset $\Hal$ of  $M(X)$ satisfying the following conditions: 
\begin{enumerate}
\item[(A)] If $u\in \Hal$, $v\in \Hal $ and $\ell(u)<\ell(v)$, then $u<v$ in the total order. 
\item[(B)] $X\subseteq \Hal$ and $\Hal \cap X_2$ consists of the products $[x,y]$ with $x, y$ in $X$ and  $x < y$. 
\item[(C)] An element $w$ of $M(X)$ of length $\geq 3$ belongs to $\Hal $ if and only if it is of the form $[a,[b,c]]$ with $a,b,c$ in $\Hal$, $[b,c]\in \Hal$, $b\leq a < [b,c]$ and $b < c$. 
\end{enumerate} 
\end{definition}
In the rest of this section, $[x,y]$ denotes the 
Lie bracket of a free Lie algebra.  Set $\Hal^i=\Hal\cap X_i$ and let $\#X=n$.  
The rank of 
$\f_n^k/\f_n^{k+1}$ is  given by Witt's formula:
\begin{equation*}
r_n(k):=\frac{1}{k}\sum_{d|k}\mu(d)n^{k/d},
\end{equation*}
where $\mu$ is the M\"{o}bius function. 
Under the natural projection
$\f_n^k\to \f_n^k/\f_n^{k+1}$,
the images of elements of $\Hal^k$ 
form a $\Z$-basis of 
$\f_n^k/\f_n^{k+1}$.
For a proof see~\cite[Chapter II]{Bourbaki} or~\cite[Theorem 4.2]{Serre-Lie-Algebra}. 
 Let $\f_{n,c}:=\f_{n,c}(\Z)$ be the free nilpotent $\Z$-Lie algebra on $n$ generators  and of class $c$; it is defined to be the quotient algebra $\f_n/\f^{c+1}_{n}$.   
 The following facts are well known. Since we do not know a reference and their proofs are easy, we outline the arguments.  
\begin{proposition}\label{free-properties} For $n \ge 2$ and $c \ge 2$, 
\begin{enumerate}
\item[1)] The image of $\bigcup_{i=1}^{c} \Hal^i$ 
under the natural projection $\f_n\to \f_{n,c}$
 is a basis of $\f_{n,c}$.
\item[2)] The image of $\bigcup_{i=2}^{c}  \Hal^i$ 
under the natural projection $\f_n\to \f_{n,c}$
is a basis of $\f'_{n,c}$.
\item[3)] The image of $\Hal^{c}$ 
under the natural projection $\f_n\to \f_{n,c}$
is a basis of $\ZZ(\f_{n,c})$.
\end{enumerate}
\end{proposition}
\begin{proof}
The first statement follows from the fact that 
$\bigcup_{i=1}^\infty\Hal^i$ is a basis of $\f_n$, and 
the $\Z$-submodule $\f_n^{c+1}$ 
of $\f_n$
is generated by $\bigcup_{i=c+1}^\infty\Hal^i$. For the second statement, it is now enough to note that  
$\bigcup_{i=2}^{c}  \Hal^i$ generates the $\Z$-module $\f_{n,c}'$. The proof of the third statement is similar. 
\end{proof}
\begin{example}\label{Example-f_{33}} 
The image of 
$\bigcup_{i=1}^3\mathcal H^i$  
under the natural projection $\f_3\to \f_{3,3}$
 is a basis of $\f_{3,3}$. The elements of this union are explicitly given as follows. 
\begin{align*}
&\mathcal{H}^1:\, x_1,\, x_2,\, x_3,\\[.1cm] 
&\mathcal{H}^2:\, \ww_9=[x_1,x_2],\quad \ww_{10}=[x_1,x_3],\quad \ww_{11}=[x_2,x_3],\\[.1cm] 
&\mathcal{H}^3:\, \ww_1=[x_1,[x_1,x_2]],\quad \ww_2=[x_1,[x_1,x_3]],\quad \ww_3=[x_2,[x_1,x_2]],\quad \ww_4=[x_2,[x_1,x_3]],\\ 
&\qquad\,\, \ww_5=[x_2,[x_2,x_3]],\quad \ww_6=[x_3,[x_1,x_2]],\quad \ww_7=[x_3,[x_1,x_3]],\quad \ww_8=[x_3,[x_2,x_3]].
\end{align*}
Note that $\ZZ(\f_{3,3})\subseteq \f_{3,3}'$ and furthermore one can check that the image of $\{\ww_1,\dots, \ww_8\}$ is a basis of $\ZZ(\f_{3,3})$, while the image of $\{\ww_1,\dots,\ww_{11}\}$ is a basis of $\f_{3,3}'$.
By an explicit calculation and using the Jacobi identity, we obtain the following  commutator matrix
\begin{equation*}
F_{\f_{3,3}}(T_1,\dots,T_{11})=
\begin{footnotesize}
\left( {\begin{array}{*{20}c}
   {0} & {T_9 } & {T_{10} } &\vline &  {T_1 } & {T_2 } & {T_4  - T_6 }  \\
   { - T_9 } & {0} & {T_{11} } &\vline &  {T_3 } & {T_4 } & {T_5 }  \\
   { - T_{10} } & { - T_{11} } & {0} &\vline &  {T_6 } & {T_7 } & {T_8 }  \\
\hline
   { - T_1 } & { - T_3 } & { - T_6 } &\vline &  {0} & {0} & {0}  \\
   { - T_2 } & { - T_4 } & { - T_7 } &\vline &  {0} & {0} & {0}  \\
   {T_6  - T_4 } & { - T_5 } & { - T_8 } &\vline &  {0} & {0} & {0}  \\

 \end{array} } \right)
\end{footnotesize} 
 =\begin{pmatrix}
F_{11} & F_{12}\\[0.2cm]
-F_{12}^{\mathrm{tr}} & 0
\end{pmatrix}.
\end{equation*}
\end{example}
Now we consider the general case. Set 
$$m:=\sum_{k=2}^{c} r_n(k)\ \text{ and }\ m_1:=\sum_{k=1}^{c-1} r_n(k).$$
By Proposition~\ref{free-properties}, the natural projection maps $\cup_{i=1}^{c-1}  \Hal^i$  to a basis of $\f_{n,c}/(\ZZ(\f_{n,c}))$. Note that the commutator matrix $F_{\f_{n,c}}({\bf T})\in M_{m_1}(\Z[{\bf T}])$ is a skew-symmetric matrix whose entries are $\Z$-linear forms in $m$ variables. 
We label the variables as follows. For each $2\leq k\leq c$ we write ${\bf T}^{(k)}=(T_1^{(k)},\dots, T_{r_n(k)}^{(k)})$,
so that ${\bf T}=({\bf T}^{(k)})_{2\leq k\leq c}$ and 
\begin{equation*}
F_{\f_{n,c}}({\bf T})=
\begin{footnotesize}
\begin{pmatrix}
F_{11}({\bf T}^{(2)}) & F_{12}({\bf T}^{(3)}) & \dots & F_{1(c-1)}({\bf T}^{(c)})\\[.2cm]
F_{21}({\bf T}^{(3)}) & F_{22}({\bf T}^{(4)}) & \dots & 0\\
\vdots & \vdots & \vdots & \vdots\\
F_{(c-1)1}({\bf T}^{(c)}) & 0 & 0 & 0\\
\end{pmatrix}
\end{footnotesize},
\end{equation*}
where $F_{ij}({\bf T}^{(i+j)})$ is the zero matrix if $i +j > c$, and  $F_{ij}({\bf T}^{(i+j)})=-F_{ji}^{\mathrm{tr}}({\bf T}^{(i+j)})$.

In order to 
use Theorem~\ref{Algorithm}
to compute the faithful dimension of $\exp(\f_{n,c}\otimes_\Z\F_q)$ we need to find $r_n(c)$ vectors ${\bf a}_\ell=({\bf a}_\ell^{(k)})$, ${2\leq k\leq c}$, with ${\bf a}_\ell^{(k)}=(a_{\ell 1}^{(k)},\dots, a_{\ell r_n(k)}^{(k)})$ such that the vectors ${\bf a}_\ell$ minimize 
$$
\sum_{\ell=1}^{r_n(c)} fq^{\frac{\rank_{\F_q}(F_{\f_{n,c}}({\bf a}_\ell))}{2}},
$$
subject to the condition 
\begin{equation}\label{ai-center}
\begin{footnotesize}
\begin{pmatrix}
a_{1 1}^{(c)} & a_{1 2}^{(c)} & \dots & a_{1 r_n(c)}^{(c)}\\
\vdots &  &  & \vdots\\[.1cm]
a_{r_n(c) 1}^{(c)} & a_{r_n(c) 2}^{(c)} & \dots & a_{r_n(c) r_n(c)}^{(c)}
\end{pmatrix}
\end{footnotesize}
\in \GL_{r_n(c)}(\F_q).
\end{equation}
Let us define the {\it reduced commutator matrix} of $\f_{n,c}$
to be 
\begin{equation*}
F^{\mathrm{red}}_{\f_{n,c}}({\bf T_c})=
\begin{footnotesize}
\begin{pmatrix}
0 & 0 & \dots & 0 &  F_{1(c-1)}({\bf T}^{(c)})\\[.2cm]
0 & 0 & \dots & F_{2(c-2)}({\bf T}^{(c)}) & 0\\
\vdots & \vdots & \vdots & \vdots & \vdots\\
0 & F_{(c-2)2}({\bf T}^{(c)}) & 0 & \cdots & 0 \\
F_{(c-1)1}({\bf T}^{(c)}) & 0 & 0 & \cdots & 0\\
\end{pmatrix}
\end{footnotesize},
\end{equation*}
In other words,  $F^{\mathrm{red}}_{\f_{n,c}}$ is the matrix obtained
from $F_{\f_{n,c}}$
 by setting the variables ${\bf T}^{(k)}$ equal to zero when $k\neq c$, so that the $(i,c-i)$-block of 
$F^{\mathrm{red}}_{\f_{n,c}}$
is equal to 
$F_{i(c-i)}({\bf T}^{(c)})$ for $1\leq i\leq c-1$, and all other blocks
in $F^{\mathrm{red}}_{\f_{n,c}}$
 are equal to zero.
For instance in Example~\ref{Example-f_{33}}, the reduced commutator matrix is 
$$
\begin{pmatrix}
0 & F_{12}\\
-F_{12}^{\mathrm{tr}} & 0
\end{pmatrix}.$$
 
Note that for each $2\leq k\leq c$, the variables ${\bf T}^{(k)}$ only occur in the matrices $F_{ij}$ with $i+j=k$. 
Clearly,
\begin{equation}
\label{eq:rkFqall}
\rank_{\F_q}(F_{\f_{n,c}}({\bf a}_\ell))\geq \sum_{i+j=c}\rank_{\F_q}(F_{ij}({\bf a}_\ell^{(c)})).
\end{equation}
Note that the only entries of ${\bf a}_\ell$ that appear in the invertibility condition 
\eqref{ai-center} are those of $\ba_\ell^{(c)}$. Further, by setting all of the components of 
$\ba_\ell^k$ to zero for $2\leq k<c$, we do not increase the rank of the matrix  $F_{\f_{n,c}}({\ba}_\ell)$.
Therefore the minima of the two sides of \eqref{eq:rkFqall} are equal.
Thus from Theorem~\ref{Algorithm} we can conclude the following proposition.
\begin{proposition}\label{algorithm-free} Let $F:=F^{\mathrm{red}}_{\f_{n,c}}$ be the reduced commutator matrix of $\f_{n,c}$. Then the faithful dimension of $\exp(\f_{n,c}\otimes_\Z\F_q)$ is
\begin{equation*}
\min\left\{ \sum_{\ell=1}^{r_n(c)}  fq^{\frac{\rank_{\F_q}\left(F\left(a_{\ell 1},\dots,a_{\ell r_n(c)}\right)\right)}{2}}: \begin{pmatrix}
a_{1 1} & \dots & a_{1 r_n(c)}\\
\vdots  & \ddots & \vdots\\
 a_{r_n(c) 1} & \dots & a_{r_n(c) r_n(c)}
\end{pmatrix}\in\mathrm{GL}_{r_n(c)}(\F_q)\right\}.
\end{equation*}
\end{proposition}

\subsection{Proof of Theorem~\ref{thm:free}} We now prove Theorem~\ref{thm:free}. The proof relies upon an explicit description of the commutator matrix and so it is combinatorial in nature.  
First we consider the statement for $\f_{n,2}$. 
The image of the set  
$$\Hal^1\cup\Hal^2= \{x_1,\dots, x_n, x_{ij}: 1\leq i<j\leq n\},$$
where $x_{ij}:=[x_i,x_j]$ for $1\leq i<j\leq n$, is
a basis of  $\f_{n,2}$. 
Thus
the commutator matrix of $\f_{n,2}$ is 
\begin{equation*}
F_{\f_{n,2}}({\bf T})=
\begin{footnotesize}
\begin{pmatrix}
0 & T_{12} & T_{13}& \dots & T_{1n}\\
-T_{12} & 0 & T_{23} & \dots & T_{2n}\\
-T_{13} & -T_{23} & 0 & \dots & T_{3n}\\
\vdots & \vdots & \vdots & \ddots & \vdots\\
-T_{1n} & -T_{2n} & -T_{3n} & \dots & 0
\end{pmatrix}
\end{footnotesize}.
\end{equation*}
Observe that each variable $T_{ij}$, $1\leq i<j\leq n$, appears exactly twice. Therefore if exactly one of the $T_{ij}$'s is non-zero, then the rank of the above matrix will be equal to 2. Now by applying Proposition~\ref{algorithm-free},  
we obtain the statement of the theorem for $\f_{n,2}$.  
 
Let us now turn to the case of $\f_{n,3}$. First note that the reduced commutator matrix of $\f_{n,3}$ is equal to
$$
F({\bf T}^{(3)}):=F_{\f_{n,3}}^{\mathrm{red}}({\bf T}^{(3)})=
\begin{footnotesize}
\begin{pmatrix}
0 & F_{12}({\bf T}^{(3)})\\
-F_{12}^{\mathrm{tr}}({\bf T}^{(3)}) & 0
\end{pmatrix}
\end{footnotesize}
.
$$
 In particular, 
$\rank_{\F_q}(F({\bf T}^{(3)}))=2\rank_{\F_q}(F_{12}({\bf T}^{(3)}))$.
Our goal is to find  $F_{12}({\bf T}^{(3)})$ explicitly with respect to the basis obtained by the image of $\Hal^1\cup\Hal^2\cup\Hal^3$. The latter set consists of
\begin{equation*}
\begin{split}
\Hal^1 &=\left\{x_i: 1\leq i\leq n\right\};\\
\Hal^2&=\left\{x_{ij}:=[x_i,x_j]: 1\leq i<j\leq n\right\};\\
\Hal^3&=\left\{x_{ijk}:=[x_i,[x_j,x_k]]: 1\leq j\leq i\leq n, \quad 1\leq j<k\leq n\right\}.
\end{split}
\end{equation*}
Consider the sets 
$R_1 := \{ 1,2, \dots, n \},$ 
$R_2:= \{ (i,j): 1 \le i < j \le n \},$
and 
\begin{equation*}
R_3:= \{ (i,j,k): 1 \le j \le i \le n, 1 \le j <k \le n \}.
\end{equation*}
The elements of the $R_i$'s  parametrize the 
sets $\Hal^1,\Hal^2$ and $\Hal^3$.
Set \[
d:=\# R_3=r_n(3)=(n^3-n)/3,
\] and define
\begin{equation*}
\begin{split}
R_3^0 &:= \{ (i,j,i): 1\leq j<i\leq n \} \cup \{ (i,i,k): 1\leq i<k\leq n \}; \\
R_3^+ & :=  \{ (i,j,k) \in R_3: j<i<k \}; \\
R_3^- &:= \{ (i,j,k) \in R_3: j<k<i \}.
\end{split}
\end{equation*}
It is clear that $R_3= R_3^0 \cup R_3^+ \cup R_3^-$ is a partition of $R_3$. 
We will now give an explicit description of $F_{12}({\bf T}^{(3)})$ in terms of these sets. It will be more convenient to use the variables $T_{ \alpha}$ with $ \alpha \in R_3$ instead of ${\bf T}^{(3)}=(T_1^{(3)},\dots, T_{r_n(3)}^{(3)})$. We will also use $\bf T$ to
denote the vector with entries $T_\alpha$ with $\alpha\in R_3$. For instance $T_{213}$ and $T_{312}$ correspond, respectively, to $T_4$ and $T_6$ in Example~\ref{Example-f_{33}}.
For $i \in R_1$ and $(j,k) \in R_2$,
 a simple computation shows that   
\begin{equation}\label{commutators}
[x_i, [x_j, x_k] ]= \begin{cases}  x_{ijk}   & \textrm{if } \,  (i,j,k) \in R_3;  \\
x_{jik}- x_{kij} & \textrm{otherwise}. \\   
\end{cases} 
\end{equation}
Therefore the entry of the matrix $F_{12}$ in the row associated to $i\in R_1$ and column associated to 
$(j,k)\in R_2$ is given by  $T_{ijk} $ if $ (i,j,k) \in R_3 $ and by $T_{jik}- T_{kij}$, otherwise. 
\begin{lemma}\label{sparse}
For $1 \le i,j,k \le n$ the following holds:
 \begin{enumerate}
\item[(a)] For $i<k$, the variable $T_{iik}$ appears exactly once in $F_{12}(\T)$, namely, in  row $i$ and column $(i,k)$.

\item[(b)] For $j<i$, the variable $T_{iji}$ appears exactly once in $F_{12}(\T)$, namely, in  row $i$ and column $(j,i)$.

\item[(c)] For $(i,j,k) \in R^+_3$, the variable $T_{ijk}$ appears exactly twice in the entries of $F_{12}(\T)$. Namely, the entry in row $i$ and column $(j,k)$  is equal to $T_{ijk}$, and the entry  in row $j$ and column $(i,k)$ is equal to $T_{ijk}-T_{kji}$.

\item[(d)] For $(i,j,k) \in R^-_3$,  the variable $T_{ijk}$ appears exactly twice in $F_{12}(\T)$. Namely, the entry in row 
$i$ and column  $(j,k)$ is equal to $T_{ijk}$, and
the entry in row  $j$ and column $(k,i)$ is equal to $T_{kji}-T_{ijk}$. 

\end{enumerate}
 \end{lemma}
\begin{proof}[Proof of the Lemma~\ref{sparse}]
Parts (a) and (b) of the lemma are clear, since if $T_{ijk}$ appears twice then $i,j,k$ must be pairwise distinct. 
We will now prove part (c). Suppose $j<i<k$. It is clear from~\eqref{commutators} that $T_{ijk}$ can only potentially appear in a bracket of the form $[x_{\alpha}, [x_{\beta}, x_{\gamma}]]$, where the indices are permutations of $i,j,k$ and 
$ \beta< \gamma$. This leaves three possibilities $( \alpha, \beta, \gamma)= (i, j, k),\, (j, i, k),\, (k, j, i)$. However, since
$[x_k, [x_j, x_i]]=x_{kji} \in \Hal^3$, the third possibility does not occur. Moreover, 
$$[x_j,[x_i,x_k] ]= [x_i, [x_j, x_k]]- [x_k, [x_j, x_i] ]= x_{ijk}-x_{kji},$$ which proves the statement. Part (d) can be proven in a similar way. 
\end{proof} 
From this it follows that $\rank_{\F_q}(F_{12}({\bf a}))\geq 1$ for every non-zero vector ${\bf a}\in\F_q^{d}$, where $d:=r_n(3)$. 
In the rest of this section, for each $\alpha\in R_3$, we will find a vector ${\bf a}_\alpha=(a_{1,\alpha},\dots,a_{d,\alpha})\in \F_q^d$ such that  
the rank of $M_\alpha:=F_{12}({\bf a}_\alpha)$ is equal to $1$ and   the $M_\alpha$'s  are linearly independent matrices over $\F_q$. It follows that the ${\bf a}_\alpha$'s are linearly independent over $\F_q$ and thus the matrix $({\bf a}_\alpha)_{\alpha\in R_3}$ is invertible. Then from Proposition~\ref{algorithm-free} we conclude that the faithful dimension of $\exp(\f_{n,3}\otimes_\Z\F_q)$ is equal to $r_n(3)fq$ when $p\geq 5$. 

We now construct $M_\alpha$ for $\alpha\in R_3$. For every $ \delta=(i,j,k)$ with $1\leq i,j,k\leq n$, define 
$$ \delta^+:= \max\{i,j,k\}, \quad  \delta^-:=\min\{i,j,k\}, \quad \delta^0:= i+j+k- \delta^+-\delta^-. $$

\noindent{\bf I)}: Suppose $ \alpha=(i,j,k) \in R_3^0$, that is $i=k$ or $i=j$. In either case, $T_{ijk}$ appears only once in $F_{12}(\T)$ by parts (a) and (b) of Lemma~\ref{sparse}. Let $M_{\alpha}$ be the matrix
obtained from $F_{12}({\bf T})$ by setting $T_{ijk}=1$ and letting the rest of variables to be zero. Then the rank of $M_\alpha$ is $1$. 

\noindent{\bf II)}: Let $ \alpha=(i,j,k) \in R_3^+$, that is $j<i<k$. Let $M_{\alpha}$ be the matrix obtained from $F_{12}({\bf T})$ by  setting $T_{ijk}= T_{iik}=T_{jjk}=1$, and zero for the rest of variables. From part (c) of Lemma~\ref{sparse} we can see that in this case:
\[ M_{\alpha}= 
\begin{footnotesize}    
\begin{blockarray}{cccccc}
 &   & _{(j,k)} &  & _{(i,k)}  & \\
\begin{block}{c(ccccc)}
    &  & \vdots &           &  \vdots &      \\
_{j}   & \dots  & T_{jjk} & \dots  & T_{ijk}- T_{kji} & \dots \\
   &   & \vdots &   & \vdots &  \\
_{i}   & \dots  & T_{ijk} & \dots  & T_{iik} & \dots \\
    &  & \vdots &           &  \vdots &      \\
\end{block}
\end{blockarray} =
\begin{blockarray}{cccccc}
 &   & _{(j,k)} &  & _{(i,k)}  & \\
\begin{block}{c(ccccc)}
    &   & \vdots &           &  \vdots &      \\
_{j}   & \dots  & 1& \dots  & 1 & \dots \\
   &   & \vdots&   & \vdots &  \\
_{i}   & \dots  & 1& \dots  & 1 & \dots \\
    &   & \vdots &           &  \vdots &      \\
\end{block}
\end{blockarray}
\end{footnotesize}\, .
\]
Note that this matrix has rank $1$ and the rows in which non-zero entries are located correspond to the $ \alpha^-$ and $ \alpha^0$. 

\noindent{\bf III)}:
Let $\alpha=(i,j,k) \in R_3^-$, that is $  j<k<i $. Let $M_{\alpha}$ be the matrix obtained from $F_{12}({\bf T})$ by  setting $T_{jjk}= T_{ ijk}=1$ and $ T_{iki}=-1$ and zero elsewhere. One can verify that in this case
\[ M_{\alpha}=
\begin{footnotesize}     
\begin{blockarray}{cccccc}
 &   & _{(j,k)} &  & _{(k,i)}  & \\
\begin{block}{c(ccccc)}
    &  & \vdots &           &  \vdots &      \\
_{j}   & \dots  & T_{jjk} & \dots  & T_{kji}- T_{ijk} & \dots \\
   &  & \vdots &  & \vdots & \\
_{i}   & \dots  & T_{ijk} & \dots  & T_{iki} & \dots \\
    &  & \vdots &           &  \vdots &      \\
\end{block}
\end{blockarray} =
\begin{blockarray}{cccccc}
 &   & _{(j,k)} &  & _{(k,i)}  & \\
\begin{block}{c(ccccc)}
    &  & \vdots &           &  \vdots &      \\
_{j}   & \dots  & 1& \dots  & -1 & \dots \\
   &   & \vdots&   & \vdots &  \\
_{i}   & \dots  & 1& \dots  & -1 & \dots \\
    &   & \vdots &           &  \vdots &      \\
\end{block}
\end{blockarray}
\end{footnotesize}\, .
\]
Note that this matrix has rank $1$ and the rows in which non-zero entries are located correspond to the $ \alpha^-$ and $ \alpha^+$. 

Let us show that the matrices constructed above are linearly independent. Suppose that  
\[ M= \sum_{\alpha  \in R_3} c_{ \alpha} M_{ \alpha}=0. \]
We will show that $c_{ \alpha}=0$ for all $ \alpha \in R_3$. Let $\beta=(s,r,t)\in R_3^+$ be an arbitrary element. 
 The definition of $R_3^+$ implies that  $r<s<t$.
Consider the entry of $M$  associated to row $s$ and  column $(r,t)$. 
Note that since $s= \beta^0$, a case by case verification  
shows that the  $(s, (r,t))$ entry of every matrix $M_{ \alpha}$ with $ \alpha \in R_3^0\cup R_3^-$ is equal to zero.  
Furthermore, the only $ \alpha \in R_3^+$ for which $M_{ \alpha}$ has a non-zero entry in row $s$ and column $(r,t)$ 
is $ \alpha=\beta$. This gives $ c_\beta=0$, which, in turn, shows that $c_{ \beta}=0$ for all $ \beta \in R_3^+$.

Now let $\gamma=(t',r',s')\in R_3^-$ be an arbitrary element. Then from the definition of $R_3^-$ we have $r'<s'<t'$. Consider the entry of $M$ associated to row $t'$ and column $(r',s')$. The $(t',(r',s'))$ entry of every matrix $M_{ \alpha}$ with $ \alpha \in R_3^0$ is equal to zero. Furthermore, 
the only $ \alpha \in R_3^-$ for which $M_{ \alpha}$ has a nonzero entry in row $t'$ and column $ (r',s')$ is
$ \alpha=\gamma$. Comparing coefficients yields $ c_\gamma=0$.  This shows that $c_\gamma=0$ for all 
$ \gamma \in R_3^-$. Hence 
\[ M=  \sum_{\alpha  \in R_3^0} c_{ \alpha} M_{ \alpha}=0. \]
It is however clear that for  distinct $ \alpha_1 , \alpha_2 \in R_3^0$ non-zero entries of $ M_{ \alpha}$ do not overlap. 
This implies that $ c_{ \alpha}=0$ for all $ \alpha \in R_3^0$, and the proof is complete.

\subsection{Proof of Theorem~\ref{thm:free-meta}}  We first note that, similar to the case of $\f_{n,c}$, one can define the {\it reduced commutator matrix} of $\m_{n,c}$ and prove the same result as Proposition~\ref{algorithm-free}. Let us now turn to the Lie algebra $\m_{2,c}$ generated by $x_1$ and $x_2$. Note that since $\m_{2,c}$ is metabelian, the only elements of the Hall basis whose images in  $\m_{2,c}$ are non-zero are of the form 
\begin{equation}\label{long}
[x_{i_k},[x_{i_{k-1}}, \dots, [x_{i_1}, x_{2}] \dots ],
\end{equation} where  $ k \le c-1$ and $1=i_1\leq i_2 \le \cdots \le i_k\leq 2$. Moreover,  the images in $\m_{2,c}$ of the words
in~\eqref{long}
  with $k=c-1$
form a basis of the centre of $\m_{2,c}$. Write $y^{k}_\ell$ for
{the image in $\m_{2,c}$ of the 
 unique element of the Hall basis of the form~\eqref{long} of length $k$ in which the generator $x_2$ occurs $\ell$ times. For instance, 
$y^3_1=[x_1,[x_1, x_2]]$ and $y^3_2=[x_2, [x_1, x_2]]$. Then
the image of $\{y^c_1,\dots,y^c_{c-1}\}$ is a basis of $\m_{2,c}^{c}=\ZZ(\m_{2,c})$ and 
the image of
$\{y^{c-1}_1,\dots,y^{c-1}_{c-2},y^c_1,\dots,y^c_{c-1}\}$ is a basis of $\m_{2,c}^{c-1}$. 
The fact that $\m_{2,c}$ is metabelian implies that 
\begin{equation}
y^{k+1}_{\ell}=[x_1, y^k_\ell],\qquad 
y^{k+1}_{\ell+1}=[x_2, y^k_\ell].
\end{equation}
Thus the reduced commutator matrix of $\m_{2,c}$ is of the form 
$$
\begin{pmatrix}
0 & F\\
-F^{\mathrm{tr}} & 0
\end{pmatrix},
$$  
where $F$ is a $2\times (c-2)$ matrix of the form
\[ F(T_1, \dots, T_{c-1})= \begin{pmatrix}
T_1 & T_2   & \dots & T_{c-2}   \\
T_2  & T_3   & \dots & T_{c-1}
\end{pmatrix}.
\]
Note that 
$\rank_{\F_q}(F({\bf a}))\geq 1$
for every non-zero vector ${\bf a}\in\F_q^{c-1}$,   and the matrix $F(T_1, \dots, T_{c-1})$ has rank $1$ if we set $T_i= \lambda^{i-1}$, where $ \lambda \in \F_q$.
Since  $q\geq p>c$,  we can find at least $c-1$ distinct elements $ \lambda_1, \dots, \lambda_{c-1}$ in $\F_q$. Consider the $(c-1)\times (c-1)$ matrix with the $i$th row given by 
\[ (T_1, \dots, T_{c-1})= (1, \lambda_i, \dots,  \lambda_i^{c-2}). \]
This is the well-known Vandermonde matrix, whose determinant is non-zero. 
 Similar to the proof of Proposition \ref{algorithm-free}, the 
 claim follows from Theorem~\ref{Algorithm}.

\begin{remark}
In the above proof,  the issue of finding $c-1$ matrices of rank equal to $1$ is intimately related to finding points in general position on the 
rational normal curve obtained as the image of the  Veronese map  given by 
$$\nu_{c-2}: \mathbf{P}^1(\F_q)\to \mathbf{P}^{c-2}(\F_q),\quad [X_0:X_1]\mapsto [X_0^{c-2}: X_0^{c-3}X_1:\dots:X_1^{c-2}]. $$ 
It is likely that in the cases corresponding to $\m_{n,c}$ (for $n>2$) and $\f_{n,c}$ (for $n \ge 2$ and $c>3$), the faithful dimension can be computed  using tools from algebraic geometry. 
\end{remark}


\subsection{Outline of the argument for Remark~\ref{thm:f2c}}
\label{subsec-outlineTable}
We will only consider the case $c=6$, since the other cases are similar (and the calculations are a bit simpler).  For $\f:=\f_{2,6}$, the reduced commutator matrix is a block matrix of the form:
$$
F(x_1,\dots,x_9)=
\footnotesize
\begin{pmatrix}
0 & 0 & 0 & 0 & F_{1,5}\\
0 & 0 & 0 & F_{2,4} & 0\\
0 & 0 & F_{3,3} & 0 & 0\\
0 & -F_{2,4}^{\mathrm{tr}} & 0 & 0 & 0\\
-F_{1,5}^{\mathrm{tr}} &0 & 0 & 0 & 0
\end{pmatrix},
$$
where 
\begin{equation}\label{eq1}
\begin{split}
F_{1,5}(x_1,\dots,x_9)&=
\begin{pmatrix}
x_1 & x_2+x_6 & x_3+2x_7-x_9 & x_4+2x_8 & x_6 & x_7+x_9\\
x_2 & x_3 & x_4 & x_5 & x_7-x_9 & x_8
\end{pmatrix};\\[.1cm]
F_{2,4}(x_1,\dots,x_9)&=
\begin{pmatrix}
x_6 & x_7 & x_8
\end{pmatrix};\\[.1cm]
F_{3,3}(x_1,\dots,x_9)&=
\begin{pmatrix}
0 & x_9\\
-x_9 & 0
\end{pmatrix}.
\end{split}
\end{equation}
For $p\geq 7$, Proposition~\ref{algorithm-free} implies that
\begin{equation}\label{eq-App}
\mf(\GG_p)=\min\left\{ \sum_{\ell=1}^{9}  p^{\frac{\rank_{\F_p}\left(F\left(a_{\ell 1},\dots,a_{\ell 9}\right)\right)}{2}}: \begin{pmatrix}
a_{1 1} & \dots & a_{1 9}\\
\vdots  & \ddots & \vdots\\
 a_{91} & \dots & a_{99}
\end{pmatrix}\in\mathrm{GL}_{9}(\F_p)\right\}.
\end{equation}
We can easily verify that $\rank_{\F_p}(F_{1,5}({\bf x}))\geq 1$ when $0\neq {\bf x}:=(x_1,\ldots,x_9)\in\mathbb{F}_p^9$.  
Also,
\[
\rank_{\F_p}(F(x_1,\dots,x_9))\geq 4
\]
whenever at least one of $x_6$, $x_7$ or $x_8$ is not zero. Similarly, 
$\rank_{\F_p}(F(x_1,\dots,x_9))\geq 6$
whenever
 $x_9\neq 0$.

Now let ${\bf x}_i:=(x_{i1},\dots,x_{i9})\in\mathbb{F}_p^9$, $1\leq i\leq 9$, be $9$ vectors with $\det(x_{ij})\neq 0$.  Thus after permuting the indices of the ${\bf x}_i$'s, we can assume that all of the diagonal entries $x_{ii}$, $1\leq i\leq 9$, are nonzero.  Hence, from~\eqref{eq-App} and the above discussion we deduce that the faithful dimension of $\GG_p$ is at least $p^3+3p^2+5p$. This dimension can be realized 
by the rows of following matrix:
\begin{equation*}\label{eq4}
\begin{pmatrix}
\,1 & 0 & 0 & 0 & 0 & 0 & 0 & 0 & 0\\[.2cm]
\,1 & \lambda_1 & \lambda_1^2 & \lambda_1^3 & \lambda_1^4 & 0 & 0 & 0 & 0\,\\[.1cm]
\,1 & \lambda_2 & \lambda_2^2 & \lambda_2^3 & \lambda_2^4 & 0 & 0 & 0 & 0\,\\[.1cm]
\,1 & \lambda_3 & \lambda_3^2 & \lambda_3^3 & \lambda_3^4 & 0 & 0 & 0 & 0\,\\[.1cm]
\,1 & \lambda_4 & \lambda_4^2 & \lambda_4^3 & \lambda_4^4 & 0 & 0 & 0 & 0\,\\[.1cm]
\,0 & 0 & \mu_1 & 3\mu_1^2 & 5\mu_1^3 & 1 & \mu_1 & \mu_1^2 & 0\,\\[.1cm]
\,0 & 0 & \mu_2 & 3\mu_2^2 & 5\mu_2^3 & 1 & \mu_2 & \mu_2^2 & 0\,\\[.1cm]
\,0 & 0 & \mu_3 & 3\mu_3^2 & 5\mu_3^3 & 1 & \mu_3 & \mu_3^2 & 0\,\\[.1cm]
\,0 & 0 & 0 & \eta & 5\eta^2 & 0 & 1 & 2\eta & 1\,
\end{pmatrix}\in \GL_9(\F_p),
\end{equation*}
where the $\lambda_i$'s and the $\mu_i$'s are distinct elements of $\F_p$.


\begin{acknowledgements}
We would like to thank Martin Bays, Emmanuel Breuillard, Tim Clausen, Jamshid Derakhshan, Martin Hils, Alan Huckleberry, Franziska Jahnke, Aleksandra Kwiatkowska, Gunter Malle, Katrin Tent and Pierre Touchard  for several useful discussions. Eamonn O'Brien and Christopher Voll brought a mistaken entry in Table~\ref{Table1} to our attention. We thank Christopher Voll for many useful comments.  
The authors would like  to thank the referee for carefully reading the manuscript and for providing numerous suggestions that substantially improved both the content and the exposition of the paper.  
\end{acknowledgements}



%

\end{document}